\documentclass[11pt]{amsart}
\usepackage[letterpaper,hmargin=1.35in,top=1.35in,bottom=1.2in]{geometry}
\usepackage{enumerate, amsmath, amssymb, amsthm}
\usepackage{xcolor}
\usepackage{orcidlink}
\usepackage{hyperref}
\hypersetup{colorlinks=true, linkcolor=blue!50!black, citecolor=green!35!black, urlcolor=blue!50!black}
\DeclareTextSymbolDefault{\textsection}{OMS}
\makeatletter
\def\l@subsection{\@tocline{2}{0pt}{2pc}{5pc}{}}
\makeatother

\numberwithin{equation}{section}

\newtheorem{Prop}[equation]{Proposition}
\newtheorem{Thm}[equation]{Theorem}
\newtheorem{Lem}[equation]{Lemma}
\newtheorem{Cor}[equation]{Corollary}
\theoremstyle{definition}
\newtheorem{Def}[equation]{Definition}
\newtheorem{Ex}[equation]{Example}
\newtheorem{Rem}[equation]{Remark}
\newtheorem{Rems}[equation]{Remarks}
\theoremstyle{definition}
\theoremstyle{plain}
\newtheorem{ThmIntro}{Theorem}

\newcommand{\F}{\mathbb{F}}
\newcommand{\Q}{\mathbb{Q}}
\newcommand{\N}{\mathbb{N}}
\newcommand{\Z}{\mathbb{Z}}
\newcommand{\C}{\mathbb{C}}
\renewcommand{\O}{\mathbb{O}}
\newcommand{\PP}{\mathbb{P}}
\newcommand{\Int}{\textnormal{Int}}
\newcommand{\IntQ}{\Int_{\Q}}
\newcommand{\olK}{\overline{K}}
\newcommand{\uE}{\underline{E}}
\newcommand{\oQp}{\overline{\Q_p}}
\newcommand{\Gal}{\mathrm{Gal}}
\newcommand{\K}{\widehat{K}}
\newcommand{\V}{\widehat{V}}
\newcommand{\W}{\widehat{W}}
\newcommand{\M}{\widehat{M}}
\newcommand{\ra}{\mathrm{ra}}
\newcommand{\rt}{\mathrm{rt}}
\renewcommand{\t}{\mathrm{t}}
\newcommand{\wW}{\widetilde{W}}
\newcommand{\wM}{\widetilde{M}}
\newcommand{\ov}{\overline{v}}

\newcommand{\oZp}{\overline{\Z_p}}

\newcommand{\br}{\mathrm{br}}

\title[Integrally closed Noetherian domains]{Geometric configuration of integrally closed Noetherian domains}

\author[G.W. Chang]{Gyu Whan Chang$^{1}$}
\address{$^{1}$ Department of Mathematics Education, Incheon National University, Incheon 22012, Republic of Korea}
\email{whan@inu.ac.kr}

\author[G. Peruginelli]{Giulio Peruginelli$^{2}$\orcidlink{0000-0001-7694-8920}}
\address{$^{2}$ Department of Mathematics ``Tullio Levi-Civita'', University of Padova, Via Trieste 63, 35121 Padova, Italy}
\email{gperugin@math.unipd.it}
\urladdr{https://orcid.org/0000-0001-7694-8920}

\date{\today}

\begin{document}

\begin{abstract}
In this paper, we completely describe the family of integrally closed Noetherian domains between $\Z[X]$ and $\Q[X]$. By a theorem of Heinzer, these are precisely the Krull domains lying between these two polynomial rings, so the problem amounts to deciding which families of discrete valuation rings of $\Q(X)$  give a Krull domain when intersected with $\Q[X]$.  We show that every such ring is a ring of integer-valued polynomials of the form $\{f\in\Q[X]\mid f(E_p)\subseteq \O_p, \text{ for every prime }p\}$, where $\O_p$ is the valuation ring of $\C_p$, the completion of an algebraic closure of $\Q_p$, and each $E_p$ is a finite union of ultrametric balls of $\C_p$. Conversely, such a ring of integer-valued polynomials is a Krull domain, and hence an integrally closed Noetherian domain, precisely when the centers of the balls of positive radius satisfy an explicit finiteness condition. The key ingredient is a parametrization of all the discrete valuation rings of $\Q(X)$ lying over $\Z_{(p)}$, residually  algebraic or residually transcendental, by Galois orbits of balls of $\C_p$. This novel approach provides a geometric understanding of this class of integrally closed domains. Furthermore, we also describe the UFDs between $\Z[X]$ and $\Q[X]$.\end{abstract}

\maketitle

\tableofcontents

\section{Introduction}\label{sec:intro}
\addtocontents{toc}{\protect\setcounter{tocdepth}{1}}

The class of Noetherian domains is of the utmost importance in all areas of mathematics and especially in commutative algebra and algebraic geometry. In this paper we classify all integrally closed Noetherian domains R such that $\Z[X]\subseteq R\subseteq\Q[X]$. Our starting point is a theorem of Heinzer \cite[Theorem 9]{Heinzer}: a Krull overring of a two-dimensional Noetherian domain is Noetherian. Conversely, it is a classical result that an integrally closed Noetherian domain is Krull. Since $\Z[X]$ is a two-dimensional Noetherian domain, the integrally closed Noetherian domains between $\Z[X]$ and $\Q[X]$ are precisely the Krull domains between these two polynomial rings, and it is the latter class that we describe. We accomplish this result by showing that these rings can be represented in a suitable way as rings of integer-valued polynomials on finite unions of ultrametric balls (Theorems A and B).

Following \cite{Chang}, a subring of $\Q[X]$ containing $\Z[X]$ is called a \emph{polynomial overring} of $\Z$, and it is called a \emph{polynomial Dedekind domain} (resp., a \emph{polynomial Krull domain}) if it is a Dedekind (resp., Krull) domain.  In order to describe the polynomial Krull domains, we first classify the DVRs of $\Q(X)$, and then we determine when a family of such DVRs has finite character, so that their intersection is a Krull domain; the relevant background on Krull domains is recalled in \S\ref{subsec:Krull}.

{{\subsection*{Main results}

Throughout, for a prime $p$, $\C_p$ denotes the completion of a fixed algebraic closure $\oQp$ of the field $\Q_p$ of $p$-adic numbers, $v_p$ is the $p$-adic valuation on $\C_p$, $\O_p$ is its valuation ring and $G_p=\Gal(\oQp/\Q_p)$, which acts on $\C_p$ by continuous automorphisms. For $\alpha\in\C_p$ and $\delta\in\Q\cup\{\infty\}$, we consider the ultrametric ball $\bar B_p(\alpha,\delta)=\{x\in\C_p\mid v_p(x-\alpha)\geq\delta\}$, so that $\bar B_p(\alpha,\infty)=\{\alpha\}$. We denote by $\C_p^{\br}$ the set of $\alpha\in\C_p$ such that the extension $\Q_p(\alpha)/\Q_p$ has finite ramification index, and we set $\O_p^{\br}=\O_p\cap\C_p^{\br}$. Given subsets $E_p\subseteq\O_p$, we set
$$\IntQ(E_p,\O_p)=\{f\in\Q[X]\mid f(E_p)\subseteq\O_p\}\quad\text{and}\quad \IntQ(\uE,\O)=\bigcap_{p}\IntQ(E_p,\O_p),$$
where $\uE=\prod_p E_p$ and $\O=\prod_p\O_p$; note that $\IntQ(\emptyset,\O_p)=\Q[X]$.

We say that a subset $E_p$ of $\O_p$ is \emph{admissible} if it is a finite (possibly empty) union of balls $\bar B_p(\alpha_j,\delta_j)$ with $\alpha_j\in\O_p^{\br}$ and $\delta_j\in\Q_{\geq0}\cup\{\infty\}$, such that $\alpha_j$ is transcendental over $\Q$ whenever $\delta_j=\infty$, and no ball is contained in a $G_p$-conjugate of another one. The local description is the following.

\begin{ThmIntro}\label{thmA}
Let $p\in\Z$ be a prime and let $R$ be a ring with $\Z_{(p)}[X]\subseteq R\subseteq\Q[X]$. Then $R$ is an integrally closed Noetherian domain if and only if $R=\IntQ(E_p,\O_p)$ for some admissible subset $E_p$ of $\O_p$, which is unique up to the action of $G_p$ on the balls. If $E_p=\cup_{j=1}^m\bar B_p(\alpha_j,\delta_j)$, then
\begin{enumerate}
\item[\upshape (a)] $R$ is a Dedekind domain if and only if $E_p$ is finite, that is, $\delta_j=\infty$ for every $j=1,\ldots,m$;
\item[\upshape (b)] if $m\geq1$, then $\mathrm{Cl}(R)\cong\Z/d\Z\oplus\Z^{m-1}$, where $d$ is the greatest common divisor of the ramification indices over $\Z_{(p)}$ of the DVRs $\{\phi\in\Q(X)\mid v_{p,\alpha_j,\delta_j}(\phi)\geq0\}$, which are computed explicitly in terms of $(\alpha_j,\delta_j)$;
\item[\upshape (c)] $R$ is a UFD if and only if $m=1$, $\alpha_1$ can be chosen unramified over $\Q_p$ with $(\alpha_1,\delta_1)$ a minimal pair, and $\delta_1\in\Z_{\geq0}\cup\{\infty\}$.
\end{enumerate}
\end{ThmIntro}

\noindent See Theorems {\ref{ultrametric balls modulo Galois are DVRs}, \ref{thm4.23}, \ref{characterization local Krull domain}} and \ref{intersection DVR is UFD}. Here $v_{p,\alpha,\delta}$ is the monomial valuation of \S\ref{monomial valuations} if $\delta<\infty$ and $v_{p,\alpha,\infty}(\phi)=v_p(\phi(\alpha))$; Theorem \ref{characterization local Krull domain} also describes the maximal ideals of $R$ in geometric terms. Note that the balls live in $\C_p$, whose residue field is infinite: for instance, the ball $\bar B_p(0,0)$ is $\O_p$ and $\IntQ(\O_p,\O_p)=\Z_{(p)}[X]$, whereas the classical ring $\Int(\Z_{(p)})=\IntQ(\Z_p,\Z_p)$ is a non-Noetherian Pr\"ufer domain.

The global description requires a finiteness condition, which reflects the finite character property of Krull domains.

\begin{ThmIntro}\label{thmB}
Let $R$ be a ring with $\Z[X]\subseteq R\subseteq\Q[X]$. Then $R$ is an integrally closed Noetherian domain if and only if $R=\IntQ(\uE,\O)$, where $\uE=\prod_p E_p$, each $E_p\subseteq\O_p$ is admissible, and the following condition holds:
\begin{enumerate}
\item[$(\ast)$] for every nonzero $g\in\Z[X]$, there are only finitely many primes $p$ such that $v_p(g(x))>0$ for some $x$ lying in a ball of positive radius of $E_p$.
\end{enumerate}
In this case, $\mathrm{Cl}(R)\cong\bigoplus_{p}(\Z/d_p\Z\oplus\Z^{m_p-1})$, where $p$ runs over the primes with $E_p\neq\emptyset$, $m_p$ is the number of balls of $E_p$ and $d_p$ is as in Theorem \ref{thmA}\textup{(b)}.
\end{ThmIntro}

\noindent See Theorems \ref{Krull domains finite character} and \ref{main thm polynomial krull domains}, together with  Proposition \ref{contraction maximal ideal} and Remarks {\ref{ball containment} and \ref{irredundant represent balls}}. Condition $(\ast)$ is automatic if $E_p\neq\emptyset$ for only finitely many $p$, while for instance $\bigcap_{p}\IntQ(\bar B_p(0,1),\O_p)$ is not Noetherian, since $X$ lies in infinitely many height-one prime ideals (Example \ref{ex430}). In Corollary \ref{global UFD} we describe the UFDs between $\Z[X]$ and $\Q[X]$, and in Theorem \ref{thm:prescribed} we show that every direct sum of a countable family of finitely generated abelian groups is the divisor class group of an integrally closed Noetherian domain between $\Z[X]$ and $\Q[X]$ which is not Dedekind.

\medskip\noindent\textbf{Example.} Let $E_2=\bar B_2(0,1/2)$ and $E_p=\O_p$ for $p\neq2$. Condition $(\ast)$ clearly holds and
$$R=\IntQ(\uE,\O)=\{f\in\Q[X]\mid f(\O_p)\subseteq\O_p\ (p\neq2),\ f(\bar B_2(0,1/2))\subseteq\O_2\}=\Z[X,X^2/2].$$
The DVR attached to the ball $\bar B_2(0,1/2)$ has value group $\frac{1}{2}\Z$, so it is ramified over $\Z_{(2)}$ with index $2$, and by Theorem \ref{thmB} we get $\mathrm{Cl}(R)\cong\Z/2\Z$. Accordingly, $R$ is not a UFD: $X\cdot X=2\cdot(X^2/2)$ are two factorizations of $X^2$ into irreducible elements of $R$. On the other hand, replacing $\bar B_2(0,1/2)$ with $\bar B_2(0,1)$ gives $R=\Z[X/2]$, which is a UFD.
\medskip

{\subsection*{Context and previous work}

\emph{Polynomial overrings of $\Z$.} Rings between $\Z[X]$ and $\Q[X]$ have been studied from several points of view. The best known example is the ring $\Int(\Z)$ of integer-valued polynomials, a two-dimensional Pr\"ufer domain which is neither Noetherian nor a Krull domain \cite{cc97}. In \cite{Chang}, the author considered polynomial overrings of the form $R=\bigcap_{p\in\Lambda}\bigcap_{j=1}^{m_p}(W_{p,j}\cap\Q[X])$, where the $W_{p,j}$ are DVRs of $\Q(X)$ lying over $\Z_{(p)}$, and computed their ($t$-)class groups in terms of the ramification indices of the $W_{p,j}$. In \cite{PerDedekind}, the polynomial Dedekind domains whose residue fields of prime characteristic are finite are described, together with their Picard groups, and in \cite{PerStacked} the finiteness assumption was removed, showing that every polynomial Dedekind domain is a ring of integer-valued polynomials $\IntQ(\uE,\O)$ where each $E_p$ is a finite set of elements of $\C_p^{\br}$ which are transcendental over $\Q$ (see also Proposition \ref{Krull Dedekind}). In the language of this paper, polynomial Dedekind domains are exactly those whose unitary essential valuations are residually algebraic over $\Z_{(p)}$, that is, those for which every ball of every $E_p$ has infinite radius. The present paper treats the general case: we allow residually transcendental DVRs, which correspond to balls of finite radius and whose centers are height-one prime ideals contained in maximal ideals of height two, and we show that these two kinds of DVRs suffice to describe all integrally closed Noetherian domains between $\Z[X]$ and $\Q[X]$. Besides the classification itself, the new ingredients are a geometric description of the maximal ideals (Theorems \ref{characterization local Krull domain} and \ref{nonunitary prime ideals}), a criterion for finite character (Theorem \ref{Krull domains finite character}) and the description of the UFDs (Theorem \ref{intersection DVR is UFD} and Corollary \ref{global UFD}).

\emph{Valuations on $K(X)$.} {A DVR $W$ of $\Q(X)$ either contains $\Q$, in which case $W=\Q[X]_{(f)}$ for some irreducible $f\in\Q[X]$ or $W=\Q[1/X]_{(1/X)}$ \cite{Chevalley}, or lies over $\Z_{(p)}$ for some prime $p$; only the latter are relevant here. The DVRs lying over $\Z_{(p)}$ whose residue field is algebraic over $\mathbb{F}_p$ were classified in \cite{PerStacked}. More generally,} describing the extensions of a valuation $v$ of a field $K$ to the field of rational functions $K(X)$ is a classical problem. Alexandru, Popescu and Zaharescu characterized the residually transcendental extensions by means of minimal pairs \cite{AlePop,APZTheorem,APZAll}, and Nart and Novacoski gave a geometric parametrization of the valuations on a polynomial ring in one variable \cite{NN}. For $K=\Q$ and the $p$-adic valuation, we obtain a single description of all the DVRs of $\Q(X)$ lying over $\Z_{(p)}$, residually algebraic or residually transcendental, by $G_p$-orbits of balls of $\C_p$ of finite or infinite radius (Theorem \ref{ultrametric balls modulo Galois are DVRs}). Under this correspondence, inclusion of the associated rings of polynomials corresponds to reverse inclusion of balls up to conjugation (Proposition \ref{overring intersection DVRs}), so that finite intersections of such rings become rings of integer-valued polynomials on finite unions of balls; the proof relies on Chabert's theory of polynomial closure \cite{ChabPolCloVal}. We believe that this description is of independent interest.

\emph{Class groups.} By a theorem of Claborn \cite{Claborn}, every abelian group is the class group of some Dedekind domain, and Eakin and Heinzer constructed Dedekind domains with prescribed class group of a more special kind \cite{EakHei}. It is natural to ask which groups occur when the ring is required to lie between $\Z[X]$ and $\Q[X]$. For polynomial Dedekind domains, the class groups are exactly the direct sums of countably many finitely generated abelian groups \cite[Theorem 3.1]{PerDedekind}. We show that the same groups are exactly the divisor class groups of the integrally closed Noetherian domains between $\Z[X]$ and $\Q[X]$, and that each of them is already realized by a \emph{pure} Krull domain, that is, one in which no height-one prime ideal containing a prime number is maximal (Theorem \ref{thm:prescribed}). In particular, the class group does not distinguish the Dedekind domains from the integrally closed Noetherian domains of dimension two in this family.

}}

\subsection*{Outline of the paper}
{This paper consists of {six} sections including the introduction. In Section \ref{sec:prelim} we recall some background on Krull domains and on rings of integer-valued polynomials.}

Section \ref{torsion extensions} is devoted to extensions of a valuation domain $V$ with quotient field $K$ to the field of rational functions $K(X)$. In \S\ref{res transcendental} we recall the notion of monomial valuation, which is a standard way to extend $V$ to $K(X)$, and the description, due to Alexandru, Popescu and Zaharescu, of the set $\mathcal W_{\rt}$ of residually transcendental extensions of $V$ to $K(X)$; {we then determine when two monomial valuations of $\overline{K}(X)$ restrict to the same extension, in terms of ultrametric balls in $\olK$ (\S\ref{section parametrization})}, and we describe the valuation domains of $K(X)$ contained in a DVR of the form $K[X]_{(q)}$ (\S\ref{subvaluation section}). It is well known that a residually transcendental extension $W$ of $V$ to $K(X)$ is a torsion extension of $V$, i.e., the value group of $W$ is torsion over that of $V$. In \S\ref{torsion ext sub} we study the ordering of torsion extensions $W_1, W_2$ of $V$ to $K(X)$ given by $W_1\leq W_2$ if $W_1\cap K[X]\subseteq W_2\cap K[X]$, and in \S\ref{res transc completion} we show that $\mathcal W_{\rt}$ is in bijection with the set of residually transcendental extensions of the completion $\V$ of $V$ to $\K(X)$.

In Section \ref{list DVRs}, we use these results to parametrize the residually transcendental extensions of $\Z_{(p)}$ to $\Q(X)$ in terms of ultrametric balls in the algebraic closure $\oQp$ of the $p$-adic field $\Q_p$. We then unify the descriptions of the DVRs of $\Q(X)$ which are either residually transcendental or residually algebraic over $\Z_{(p)}$ by means of ultrametric balls with possibly infinite radius in $\C_p$, the completion of $\oQp$  (Theorem \ref{unitary DVRs over Z(p)}), and we show that the partial order on this set of DVRs mirrors the containment between these balls, modulo the action of the Galois group $\Gal(\oQp/\Q_p)$ (Theorem \ref{ultrametric balls modulo Galois are DVRs}).

{Readers mainly interested in the classification may read Section \ref{sec:prelim}, the statements of Proposition \ref{overring intersection DVRs} and Theorem \ref{ultrametric balls modulo Galois are DVRs}, and then go directly to Sections \ref{section local case} and \ref{section global case}; Section \ref{torsion extensions} and the remaining parts of Section \ref{list DVRs} provide the valuation-theoretic tools used in their proofs. Sections \ref{section local case} and \ref{section global case} contain our main results. In Section \ref{section local case}} we fix a prime $p\in\Z$: we show that every Krull domain between $\Z_{(p)}[X]$ and $\Q[X]$ is a ring of integer-valued polynomials on a finite union of ultrametric balls of $\C_p$ (Theorem \ref{thm4.23}); after computing the centers on $\Z_{(p)}[X]$ of the DVRs of $\Q(X)$ lying over $\Z_{(p)}$ (\S\ref{centers}), we  also describe its height-one and maximal prime ideals, the ramification indices of its essential valuations and its divisor class group (Theorems \ref{characterization local Krull domain} and \ref{nonunitary prime ideals}), and we characterize the UFDs between these two rings (Theorem \ref{intersection DVR is UFD}). {In Section \ref{section global case}} we globalize these results. Namely, we consider a set $\Lambda$ of integer primes and, for each $p\in\Lambda$, a finite set of DVRs $W_{p, j}$ of $\Q(X)$ with $X \in W_{p,j}$ and $W_{p,j} \cap \Q = \Z_{(p)}$, and we determine when $R= \bigcap_{p \in \Lambda}\bigcap_{j=1}^{m_p}(W_{p,j} \cap \Q[X])$ has finite character, so that $R$ is a Krull domain (Theorem \ref{Krull domains finite character}). This yields a complete description of the polynomial Krull domains, equivalently, of the integrally closed Noetherian domains between $\Z[X]$ and $\Q[X]$ (Theorem \ref{main thm polynomial krull domains}), and of the UFDs between these two rings (Corollary \ref{global UFD}). We also show that the divisor class group of such a ring is a direct sum of a countable family of finitely generated abelian groups and, conversely, that for any such group $G$ there exists a Krull non-Dedekind domain between $\Z[X]$ and $\Q[X]$ with divisor class group isomorphic to $G$ (\S\ref{section: construction Krull}).}

\addtocontents{toc}{\protect\setcounter{tocdepth}{2}}
\section{Preliminaries}\label{sec:prelim}

{In this section we fix the notation and recall the background used throughout the paper. In \S\ref{subsec:Krull} we recall defining families{, irredundant representations} and divisor class groups of Krull domains, and in \S\ref{subsec:intval} we recall rings of integer-valued polynomials and the notion of polynomial closure, which provide the language in which our main results are stated.}

\subsection{{Krull domains}}\label{subsec:Krull}

{Let $D$ be an integral domain with quotient field $K$ and let $X^1(D)$ be the set of height-one prime ideals of $D$. Recall that $D$ is a \emph{Krull domain} if $D=\bigcap_{P\in X^1(D)}D_P$, each $D_P$ is a DVR, and this intersection has \emph{finite character}, that is, each nonzero element of $D$ is a unit in $D_P$ for all but finitely many $P\in X^1(D)$. Equivalently, $D=\bigcap_{V\in F}V$ for some family $F$ of DVRs with quotient field $K$ such that each nonzero element of $D$ is a unit in all but finitely many $V\in F$ \cite[Proposition 3]{n55}; such an $F$ is called a \emph{defining family} of $D$, and every defining family contains $\{D_P\mid P\in X^1(D)\}$ (Theorem \ref{thm9}). Dedekind domains, UFDs, polynomial rings over Krull domains and integral closures of Noetherian domains \cite[Theorem 2]{n55} are Krull domains, and a Krull domain is a Dedekind domain if and only if its dimension is at most one; see \cite{Fossum,Gilm}. Finite character cannot be dropped: the ring $E$ of entire functions satisfies $E=\bigcap_{Q\in X^1(E)}E_Q$ with each $E_Q$ a DVR, but $E$ is not a Krull domain.

If $D$ is a Krull domain, the \emph{divisor class group} $\mathrm{Cl}(D)$ is the quotient of the free abelian group on $X^1(D)$ by the subgroup generated by the principal divisors $\sum_{P}v_P(x)P$, $x\in K\setminus\{0\}$, where $v_P$ is the valuation associated to $D_P$; equivalently, it is the group of divisorial fractional ideals of $D$ under $v$-multiplication modulo the principal ones. A Krull domain is a UFD if and only if $\mathrm{Cl}(D)=0$, and if $D$ is a Dedekind domain, then $\mathrm{Cl}(D)$ is the ideal class group of $D$ \cite{Fossum}.}

Let $D$ be an integral domain with quotient field $K$ and suppose that $D=\bigcap\limits_{i\in I} V_i$, where $\{V_i\}_{i\in I}$ is a family of valuation domains of $K$ (by Krull's Theorem this is equivalent to {saying} that $D$ is integrally closed in $K$). As in \cite{gh68},
we say that an element $V=V_{i_0}$ of this family is superfluous if $D=\bigcap\limits_{i\in I,i\not=i_0}V_i$ and that the representation $D=\bigcap\limits_{i\in I} V_i$ is irredundant if no element $V_i$ of the family $\{V_i\}_{i\in I}$ is superfluous.

The following result is classical, see for example \cite[Remark on page 296]{n55} or \cite[Corollary 43.9]{Gilm}.

\begin{Thm} \label{thm9}
Let $D$ be a Krull domain. Then $\{D_P \mid P \in X^1(D)\}$ is
the unique defining family for $D$, that is,
if $\{V_{\alpha}\}$ is a defining family for $D$, then
$\{D_P \mid P \in X^1(D)\} \subseteq \{V_{\alpha}\}$.
\end{Thm}

Let $D$ be a Krull domain.
Theorem \ref{thm9} also says that if $\Lambda$ is a subset of $X^1(D)$,
then $D \subsetneq \bigcap\limits_{P \in \Lambda}D_P$ if and only if $\Lambda \subsetneq X^1(D)$.
In particular, for every $P\in X^1(D)$, $D_P$ is not superfluous for the defining family $\{D_P\mid P\in X^1(D)\}$,
which therefore realizes the unique irredundant representation of $D$.

\begin{Cor} \label{cor10}
Let $D$ be a Krull domain, $\{V_{\alpha}\}$ be a defining family for $D$
and $V \in \{V_{\alpha}\}$ with maximal ideal $M$.

\begin{enumerate}
\item[\upshape (1)] $V$ is superfluous in $D = \bigcap\limits_{\alpha}V_{\alpha}$
if and only if $\operatorname{ht}(M \cap D) \geq 2$.
\item[\upshape (2)] The intersection $D = \bigcap\limits_{\alpha}V_{\alpha}$ is irredundant if and only if $\operatorname{ht}(M_{\alpha} \cap D)=1$ for all $\alpha$.
\end{enumerate}
\end{Cor}

\begin{proof}
(1) Suppose that $V$ is not superfluous in $D = \bigcap\limits_{\alpha}V_{\alpha}$
and let $R = \bigcap\limits_{V_{\alpha} \neq V}V_{\alpha}$. Then $D = R \cap V$ and $R \nsubseteq V$.
Hence $V = D_{M \cap D}$ \cite[Lemma 1.3]{HeinzerOhm}, and since $V$ is a DVR,
$\operatorname{ht}(M \cap D) = \operatorname{ht}((M \cap D)D_{M \cap D}) = 1$. Conversely, assume that $\operatorname{ht}(M \cap D) < 2$.
Then $\operatorname{ht}(M \cap D)=1$, and hence $D_{M \cap D}$ is a DVR.
Note that $D_{M \cap D} \subseteq V$, so $D_{M \cap D} = V$.
Thus $V$ is not superfluous in $D = \bigcap\limits_{\alpha}V_{\alpha}$.

(2) This follows directly from (1).
\end{proof}

\subsection{Rings of integer-valued polynomials}\label{subsec:intval}

Let $D$ be an integral domain with quotient field $K$,
$X$ be an indeterminate over $D$, $D[X]$ be the polynomial ring over $D$,
and $A$ be a torsion-free $D$-algebra. By embedding $A$ and $K$ in the extended $K$-algebra $B=A\otimes_D K$,
we may evaluate a polynomial $f\in K[X]$ at an element $a\in A$, so that $f(a)\in B$. If $f(a)\in A$, we say that $f(X)$ is integer-valued at $a$.
Given a subset $S$ of $A$, we define the ring of integer-valued polynomials on $S$ as
$$\Int_K(S,A)=\{f\in K[X] \mid f(s)\in A,\forall s\in S\}{.}$$
Clearly, $D[X]\subseteq\Int_K(S, A)\subseteq K[X]$. If $S=A$, we set $\Int_K(A)=\Int_K(A,A)$, and if $A=D$, then we omit the subscript $K$, i.e., $\Int(S,D)= \Int_K(S,D)$ and $\Int(D) = \Int_K(D,D)$.

Given a valuation domain $V$ with quotient field $K$, the \emph{polynomial closure} of a subset $S$ of $K$ is defined as the largest subset $\overline{S}$ of $K$ for which
$$\Int(S,V)=\Int(\overline{S},V).$$
A subset $S$ of $K$ is \emph{polynomially closed} if $S=\overline{S}$. If $V$ has rank one, Chabert proves in \cite{ChabPolCloVal}
that the polynomially closed subsets of $K$ form a topology of closed subsets of $K$. For valuation domains of higher rank, the result is no {longer} true \cite{PSpolclo}.

It is well known that {$\Int(\Z)$} is a non-Noetherian two-dimensional Pr\"ufer domain \cite[Remarks VI.1.8]{cc97}{.}
Let $\PP$ be the set of prime numbers of $\Z$. For $p \in \PP$, let $\Z_p$ be the ring of $p$-adic integers,
$T_p = \{\alpha \in \Z_p \mid \alpha$ is transcendental over $\Q\}$, $\Z_{(p), \alpha} = \{\varphi \in \Q(X) \mid \varphi(\alpha) \in \Z_p\}$  (see also \S\ref{list DVRs}),
and $M_{(p), \alpha} = \{\varphi \in \Q(X) \mid \varphi(\alpha) \in p\Z_p\}$ for $\alpha \in T_p$.
Then $\Z_{(p), \alpha}$ is a DVR of $\Q(X)$ with maximal ideal $M_{(p), \alpha}$, and we have (see \cite[Proposition VI.2.2]{cc97})
$$\operatorname{Int}(\Z) = \bigcap_{p \in \PP}(\bigcap_{\alpha \in T_p}(\Z_{(p), \alpha} \cap \Q[X])),$$
$\operatorname{ht}(M_{(p), \alpha} \cap \operatorname{Int}(\Z)) = 1$, and $\Z_{(p), \alpha}/M_{(p), \alpha} = \Z/p\Z$.
Note that if {$R=\Int(\Z)$}, then $R = \bigcap\limits_{P \in X^1(R)} R_P$, $R_P$ is a DVR for all $P \in X^1(R)$,
but $R$ is not a Krull domain.

\section{{Extensions of a valuation domain $V$ to $K(X)$}}\label{torsion extensions}

Throughout this section, $V$ is a valuation domain with quotient field $K$, associated valuation $v$ and value group $\Gamma_v$. We denote by $\Gamma_{\ov}$ the divisible hull of $\Gamma_v$, namely, $\Gamma_{\ov}=\Gamma_v\otimes_{\Z}\Q$. We study {here the residually transcendental  extensions of $V$ to the field of rational functions $K(X)$, that is, valuation domains $W$ of $K(X)$ lying above $V$ (i.e., $W\cap K=V$) such that the extension of residue fields is transcendental.} In contrast, we say that an extension $W$ of $V$ to $K(X)$ is residually algebraic if the extension of the residue fields is algebraic. We describe the former by means of monomial valuations and ultrametric balls in $\overline{K}$ (\S\ref{res transcendental}); we then compare torsion extensions through their restrictions to $K[X]$ (\S\ref{torsion ext sub}), and we show that nothing changes when $V$ is replaced by its completion (\S\ref{res transc completion}). These results are applied to the case $V=\Z_{(p)}$ in Section \ref{list DVRs}.

\subsection{Residually transcendental extensions}\label{res transcendental}

In this {subsection} we recall some known facts about residually transcendental extensions of $V$ to $K(X)$.  We denote by $\mathcal W_{\rt}$  the set of {such extensions.}

\subsubsection{Monomial valuations}\label{monomial valuations}

{In order to describe $\mathcal W_{\rt}$, we first} review the notion of monomial valuations, which is a particular class of extensions of $V$ to $K(X)$.

\begin{Def}
Let $\alpha\in K$ and $\delta$ be an element of a value group $\Gamma$ which contains $\Gamma_v$. For $f\in K[X]$ such that $f(X)=a_0+a_1(X-\alpha)+\ldots+a_n(X-\alpha)^n$, we set:
$$v_{\alpha,\delta}(f)=\inf\{v(a_i)+i\delta \mid i=0,\ldots,n\}$$
The function $v_{\alpha,\delta}$ naturally extends to a valuation on $K(X)$ (\cite[Chapt. VI, \S. 10, Lemme 1]{Bourb}),
and it is called a \emph{monomial valuation}. We denote by $V_{\alpha,\delta}$ the valuation domain of $K(X)$ associated to $v_{\alpha,\delta}$, {which is an extension of } $V$.
Note that $V_{0,0}$ corresponds to the trivial Gaussian extension. Note also that $V_{\alpha,\delta}\cap K[X]=V[\frac{X-\alpha}{c}]$, where $c\in K$ has value $\delta$ (see for example \cite[Remark 3.4]{PerPrufer}).
\end{Def}

Let $\overline{K}$ be the algebraic closure of $K$,
$u$ be an extension of $v$ to $\overline{K}$, and
$U$ be the valuation domain of $u$. Given
$(\alpha,\delta)\in\overline{K}\times\Gamma_{\overline{v}}$, we denote by $v_{\alpha,\delta}$ the restriction of $u_{\alpha,\delta}$ to $K(X)$; note that this valuation depends on the chosen extension $u$ of $v$ to $\overline{K}$, see \cite[Remark 3.9]{PerPrufer}. If $V_{\alpha,\delta}$ ($U_{\alpha,\delta}$, resp.) is the valuation domain associated to $v_{\alpha,\delta}$ ($u_{\alpha,\delta}$, resp.) then clearly we have $V_{\alpha,\delta}=U_{\alpha,\delta}\cap K(X)$.
We also say that $(\alpha,\delta)\in\overline{K}\times\Gamma_{\overline{v}}$ is a \emph{minimal pair of definition of $V_{\alpha,\delta}$} (with respect to $K$)  if for each $\beta\in\overline{K}$ such that $u(\beta-\alpha)\geq \delta$, we have $[K(\beta):K]\geq[K(\alpha):K]$.

\begin{Rem}\label{ra rt}
Suppose that $\delta>v(a)$ for every $a\in K$. We denote such $\delta$ by $\infty$. Given $\alpha\in \overline{K}$, for every $f\in K[X]$ with $f(X)=a_0+a_1(X-\alpha)+\ldots+a_n(X-\alpha)^n$ over $\olK$, we have
$$v_{\alpha,\delta}(f)=u(a_0)=u(f(\alpha))$$
so in this case $V_{\alpha,\infty}=V_{\alpha}=\{\phi\in K(X)\mid \phi(\alpha)\in U\}$.
\end{Rem}

\begin{Rem}\label{remark rt}
(1) Let $U$ be a fixed extension of $V$ to $\overline{K}$.
Given $\alpha\in\overline{K}$ and an element $\delta$ in a totally ordered group containing $\Gamma_v$,
it is well-known that $V_{\alpha,\delta}$ is a residually transcendental extension of $V$
if and only if $\delta$ belongs to $\Gamma_{\ov}$ (see for example \cite{APZTheorem}  or \cite[Lemma 3.5]{PerPrufer}).
In this case, by \cite[Theorem 2.1]{APZTheorem}, the value group of $V_{\alpha,\delta}$ has finite index over the value group $\Gamma_v$ of $V$.
In particular, if $V$ is a DVR, then a residually transcendental extension of $V$ to $K(X)$ is a DVR, also.

(2) Let $V$ be a DVR and $U$ a fixed extension of $V$ to $\olK$. Given a (minimal) pair $(\alpha,\delta)\in\overline{K}\times\Gamma_{\overline v}$,
the intersection $V_{\alpha,\delta}\cap K[X]$ is clearly a Krull domain which is not Pr\"ufer, since the valuation overring $V_{\alpha,\delta}$ is residually transcendental over $V$, see \cite[Theorem 3.15]{PerPrufer}.

It is known that if we restrict to the family of Krull domains $\{V_{\alpha,\delta}\cap K[X]\mid(\alpha,\delta)\in K\times\Gamma_{\overline v}\}$, this family has a natural ordering, namely:
\begin{equation}\label{inclusion}
V_{\alpha_1,\delta_1}\cap K[X]\subseteq V_{\alpha_2,\delta_2}\cap K[X]\Leftrightarrow \delta_1\leq\delta_2
\text{ and }v(\alpha_1-\alpha_2)\geq\delta_1
\end{equation}
(see \cite[Proposition 1.1]{APZAll}; see also \cite[Remark 3.10]{PerPrufer}), so
\begin{equation}\label{equality monomials}
V_{\alpha_1,\delta_1}\cap K[X]= V_{\alpha_2,\delta_2}\cap K[X]\Leftrightarrow \delta_1=\delta_2 \text{ and } v(\alpha_1-\alpha_2)\geq \delta_1
\end{equation}
We will see in Lemma \ref{ordering} that the previous conditions are indeed equivalent to $V_{\alpha_1,\delta_1}= V_{\alpha_2,\delta_2}$.

 If $(\alpha,\delta)\in K\times \Gamma_{\overline{v}}$ and we consider the ball $B(\alpha,\delta)=\{x\in K\mid v(x-\alpha)\geq \delta\}$, then condition \eqref{inclusion} holds if and only if $B(\alpha_2,\delta_2)\subseteq B(\alpha_1,\delta_1)$.
In the case $V=\Z_{(p)}$ for some prime $p$, we generalize this ordering to the case of pairs $(\alpha,\delta)\in\C_p\times\Q$ in {\S\ref{Polordering monomial}}.
\end{Rem}

The main result of Alexandru, Popescu and Zaharescu is the following characterization of the  residually transcendental extensions of $V$ to $K(X)$ (see \cite{AlePop,APZTheorem,APZAll}).

\begin{Thm}\label{description rt Popescu}
Let $V$ be a valuation domain with quotient field $K$ and let $U$ be a fixed extension of $V$ to $\overline{K}$.
Let $W$ be a residually transcendental extension of $V$ to $K(X)$. Then there exists a minimal pair
$(\alpha,\delta)\in\overline K\times\Gamma_{\overline{v}}$ such that $W$ is the restriction to $K(X)$
of the monomial valuation $U_{\alpha,\delta}$ of the field $\overline{K}(X)$.
\end{Thm}

Note that the valuation domain $U_{\alpha,\delta}$ of $\overline K(X)$ in Theorem \ref{description rt Popescu} is a residually transcendental extension of $U$ to $\overline{K}(X)$.
The next lemma is a generalization of \cite[Lemma 3.13]{PerPrufer} to valuation domains of arbitrary rank.

\begin{Lem}\label{rt overring of V[X]}
Let $U$ be an extension of $V$ to $\overline{K}$ and let $(\alpha,\delta)\in\overline{K}\times\Gamma_{\overline{v}}$. Then
\begin{align*}
V[X]\subset U_{\alpha,\delta}\cap K(X)
\Leftrightarrow u(\alpha)\geq0\;\text{ and }\delta\geq0.
\end{align*}
\end{Lem}
\begin{proof}
Since $U_{\alpha,\delta}\cap K(X)$ is an extension of $V$ to $K(X)$, we have $V\subset U_{\alpha,\delta}\cap K(X)$, so that $V[X]\subset U_{\alpha,\delta}\cap K(X)$ if and only if $X\in U_{\alpha,\delta}$, that is, $u_{\alpha,\delta}(X)\geq0$. Since $X=\alpha+1\cdot(X-\alpha)$, by the definition of the monomial valuation we have
$$u_{\alpha,\delta}(X)=\min\{u(\alpha),\delta\}$$
(note that this also holds when $\alpha=0$, since $u(0)=\infty$). Hence $u_{\alpha,\delta}(X)\geq0$ if and only if $u(\alpha)\geq0$ and $\delta\geq0$.
\end{proof}
\subsubsection{Conjugate monomial valuations}\label{section parametrization}

In the following, {we fix an extension $U$ of $V$ to $\overline K$; its value group is isomorphic to $\Gamma_{\overline{v}}$. In this subsection, we characterize when two monomial valuations of $\overline{K}(X)$  with respect to $U$} restrict to the same extension of $V$ to $K(X)$.

Let $G_K$ be the absolute Galois group $\operatorname{Gal}(\overline K/K)$. Note that $G_K$ is also equal to the Galois group of $\overline K(X)/K(X)$. We denote by $\mathcal{D}(U)$ the decomposition group of $U$,
namely: $\mathcal{D}(U)=\{\sigma\in G_K \mid \sigma(U)=U\}$.

\begin{Lem}\label{image rt}
Let $(\alpha,\delta)\in\overline{K}\times\Gamma_{\overline{v}}$,
with $\delta=u(c)$ for some $c\in \overline{K}$, and $\sigma\in G_K$.  Then we have
$$\sigma(U_{\alpha,\delta})=\sigma(U)_{\sigma(\alpha),\delta}.$$
In particular, if $\sigma\in \mathcal{D}(U)$, then $\sigma(U_{\alpha,\delta})=U_{\sigma(\alpha),\delta}$.
\end{Lem}
\begin{proof}
Since $\frac{X-\alpha}{c}$ is a unit in $U_{\alpha,\delta}$ and  $U_{\alpha,\delta}$ is residually transcendental over the residue field $k_U$ of $U$,
it follows that $\sigma\left(\frac{X-\alpha}{c}\right)=\frac{X-\sigma(\alpha)}{\sigma(c)}$ is a unit in $\sigma(U_{\alpha,\delta})$
and $\sigma(U_{\alpha,\delta})$ is residually transcendental over $k_{\sigma(U)}$ (note that $\sigma(U_{\alpha,\delta})$ is an extension of $\sigma(U)$ to $\overline K(X)$).
Recall that the valuation associated to $\sigma(U)=U'$ is equal to $u'=u\circ\sigma^{-1}$ (\cite[Proposition 3.2.16, (1)]{EngPre}), so $u'(\sigma(c))=\delta$.
Then the claim follows by \cite[Chapt. VI, \S 10, Proposition 2]{Bourb}. If $\sigma\in \mathcal{D}(U)$, then the last claim follows at once.
\end{proof}

We now give the main result of this section, which was proved in a similar way in \cite{NN}.
In particular, given a residually transcendental extension $V_{\alpha,\delta}$ of $V$ to $K(X)$,  where $(\alpha,\delta)\in \overline{K}\times\Gamma_{\overline v}$,
we describe the set of extensions of $V_{\alpha,\delta}$  to $\overline{K}(X)$ lying above $U$  (see also \cite[Theorem 2.2]{APZMinimal}).

\begin{Thm}\label{Conjugate Rt ext}
Let $(\alpha,\delta),(\alpha',\delta')\in\overline{K}\times\Gamma_{\overline{v}}$.
Then the following conditions are equivalent.
\begin{itemize}
\item[\upshape (1)] $U_{\alpha,\delta}\cap K(X)=U_{\alpha',\delta'}\cap K(X)$.
\item[\upshape (2)] $\sigma(U_{\alpha,\delta})=U_{\alpha',\delta'}$ for some $\sigma\in\mathcal{D}(U)$.
\item[\upshape (3)] $\delta=\delta'$ and $ u(\sigma(\alpha)-\alpha')\geq\delta$, for some $\sigma\in \mathcal{D}(U)$.
\end{itemize}
In particular, the set of extensions of a residually transcendental extension $W=U_{\alpha,\delta}\cap K(X)$ to $\overline{K}(X)$ lying above $U$ is equal to $\{U_{\sigma(\alpha),\delta}\mid \sigma\in \mathcal{D}(U)\}$.
\end{Thm}

\begin{proof}
(1) $\Rightarrow$ (2) By  \cite[Conjugation Theorem 3.2.15]{EngPre},   there exists $\sigma\in G_K$ such that $\sigma(U_{\alpha,\delta})=U_{\alpha',\delta'}$. We claim that $\sigma\in \mathcal D(U)$. In fact, let $x\in U\subset U_{\alpha,\delta}$. Then $\sigma(x)\in\sigma(U_{\alpha,\delta})\cap \overline K=U_{\alpha',\delta'}\cap \overline K=U$. Thus, (2) holds.

(2) $\Rightarrow$ (3) By Lemma \ref{image rt}, $U_{\sigma(\alpha),\delta}= \sigma(U_{\alpha,\delta})= U_{\alpha',\delta'}$.
By \cite[Proposition 1.1]{APZAll},  $\delta'=\delta$ and $u(\sigma(\alpha)-\alpha')\geq\delta$, which is condition (3).

(3) $\Rightarrow$ (1) The inequality $u(\sigma(\alpha)-\alpha')\geq\delta$ is equivalent to $\alpha'\in B_u(\sigma(\alpha),\delta)=\{x\in\olK \mid u(\sigma(\alpha)-x)\geq \delta\}$, so that $(\alpha',\delta)$ is another pair of definition for $U_{\sigma(\alpha),\delta}$ by \cite[Proposition 3]{AlePop} and thus $U_{\alpha',\delta}=U_{\sigma(\alpha),\delta}$. By  Lemma \ref{image rt}, we also have $U_{\sigma(\alpha),\delta}=\sigma(U_{\alpha,\delta})$, so $U_{\alpha,\delta}$
and $U_{\alpha',\delta}$ {contract down to the same valuation domain of $K(X)$}. Thus, (1) holds.

The last claim follows immediately from the previous part and Lemma \ref{image rt}.
\end{proof}

{We remark that condition (3) of Theorem \ref{Conjugate Rt ext} says that,} up to the action of $\mathcal{D}(U)$, a residually transcendental extension of $V$ to $K(X)$ is determined by an ultrametric ball of $\overline K$ with respect to $u$.

\subsubsection{{Valuation domains contained in $K[X]_{(q)}$}}\label{subvaluation section}

In general, we remark that if $W$ is a valuation domain of the field of rational functions $K(X)$ over a field $K$ and $W$ has dimension at least $2$, then $W\cap K=V$ is a valuation domain of $K$ which is not $K$ itself. In fact, $W\cap K=K\Leftrightarrow K\subset W$ and in this case $W$ would be a DVR (\cite[p. 3]{Chevalley}).

\begin{Prop}\label{subvaluation of DVR}
Let $W$ be a valuation domain of $K(X)$ such that $W\subset K[X]_{(q)}$ for some irreducible $q\in K[X]$. Let $V=W\cap K$. Then, there exist a root $\alpha\in\olK$ of $q(X)$ and an extension $U$ of $V$ to $K(\alpha)$ such that $W=U_{\alpha}=\{\phi\in K(X)\mid \phi(\alpha)\in U\}$.
\end{Prop}
\begin{proof}
 Let $\pi:K[X]_{(q)}\to K[X]/(q)\cong K(\alpha)$ be the canonical residue map, where $\alpha$ is the residue class of $X$ and can be identified {with} a root of $q$. Note that for $\phi\in K[X]_{(q)}$, $\pi(\phi)=\phi(\alpha)$ because $\pi$ is {a} homomorphism.

It is well-known that every valuation domain $W$ strictly contained in $K[X]_{(q)}$ is of the form $\pi^{-1}(U)$, where $U$ is a valuation domain of $K(\alpha)$; moreover, $U=\pi(W)=W/Q$, where $Q$ is the maximal ideal of $K[X]_{(q)}$.  Hence,
\begin{equation}\label{residue W}
U=W/Q=\{\phi(\alpha)\mid \phi\in W\}=\{\phi(\alpha)\in K(\alpha)\mid \phi(\alpha)\in U\}.
\end{equation}
(Actually, every element of $K(\alpha)$ is of the form $g(\alpha)$ for some $g\in K[X]$). Moreover, the maximal ideal of $U$ is equal to $M_W/Q$, where $M_W$ is the maximal ideal of $W$. In particular, $U\cap K=V$. In fact, if $\beta\in V=W\cap K$, then $\pi(\beta)=\beta$ because $\pi_{\mid K}$ is the identity. Hence, $\beta=\pi(\beta)\in U=W/Q$. {Similarly}, one can show that $M_V\subseteq M_U=M_W/Q$, the maximal ideal of $U$. Therefore, $V=U\cap K$.

Finally, by \eqref{residue W} and the above considerations,
\[W=\pi^{-1}(U)=\{\phi\in K(X)\mid \phi(\alpha)\in U\}=U_{\alpha}.\qedhere\]
\end{proof}

We remark that if we consider the completion of $K(\alpha)$ with respect to the valuation domain $U$, we may consider $\alpha$ as an element of $\overline{\widehat{K}}$, algebraic over $K$, so that we have the equality:
$$W=\{\phi\in K(X)\mid \phi(\alpha)\in \widehat{U}\}$$
where $\widehat{U}$ is the completion of $U$.

\subsection{{Torsion extensions and the polynomial ordering}}\label{torsion ext sub}

We say that an extension $W$ of $V$ to $K(X)$ is torsion if $\Gamma_w\subseteq\Gamma_{\overline{v}}$; we denote by $\mathcal W_{\t}$  the set of all torsion extensions of $V$ to $K(X)$. Note that a residually transcendental extension of $V$ to $K(X)$ is torsion by Remark \ref{remark rt},
so $\mathcal W_{\rt} \subseteq \mathcal W_{\t}$. We stress that not all the residually algebraic extensions of $V$ to $K(X)$ are torsion. For example, $\{\phi\in K(X)\mid \phi(0)\in V\}$  is a residually algebraic extension of $V$ (its residue field is the same as the residue field of $V$) which is not torsion (its one-dimensional valuation overring is $K[X]_{(X)}$; see Proposition \ref{subvaluation of DVR} and also Remark \ref{why alpha transcendental}).

\subsubsection{A partial order on the set of torsion extensions of $V$ to $K(X)$}

Let $W_1,W_2$ be in $\mathcal W_{\t}$. In \cite[page 282]{APZAll}, the following definition is given only for residually transcendental extensions (which are torsion):
$$w_1\leq w_2\Leftrightarrow w_1(f)\leq w_2(f),\forall f\in K[X]$$
where the last inequality is considered in
$\Gamma_{\overline{v}}$. More generally, we give the same definition for torsion extensions.
We remark that since any $f\in K[X]$ can be written as $\frac{g(X)}{c}$ for some $g\in V[X]$ and $c\in V$, $c\not=0$,
and the valuations $w_1,w_2$ coincide on $K$, the previous condition is equivalent to:  $w_1(g)\leq w_2(g)$ for each $g\in V[X]$.
Note that clearly we have
$$w_1=w_2\Leftrightarrow w_1(f)=w_2(f),\;\forall f\in K[X]\Leftrightarrow W_1=W_2.$$
In particular, we also say that $w_1<w_2$ if and only if $w_1\leq w_2$ and $w_1\not=w_2$.

\begin{Lem}\label{ordering}
Let $W_i \in \mathcal W_{\t}$ and $M_i$ be the maximal ideal of $W_i$ for $i=1,2$.
\begin{enumerate}
\item[\upshape (1)] The following statements are equivalent.
\begin{enumerate}
\item[\upshape (a)] $w_1\leq w_2$.
\item[\upshape (b)] $W_1\cap K[X]\subseteq W_2\cap K[X].$
\item[\upshape (c)] $M_1\cap K[X]\subseteq M_2\cap K[X]$.
\end{enumerate}
\item[\upshape (2)] $W_1=W_2$ if and only if $W_1\cap K[X]=W_2\cap K[X]$, if and only if $M_1\cap K[X] = M_2\cap K[X]$.
\item[\upshape (3)] $(\mathcal W_{\t}, \leq)$ is a partially ordered set.
\end{enumerate}
\end{Lem}

\begin{proof}
(1) (a) $\Rightarrow$ (c). Let $f\in M_1\cap K[X]$ be nonzero. Then $0 < w_1(f)\leq w_2(f)$, so $f\in M_2\cap K[X]$.
Thus $M_1\cap K[X]\subseteq M_2\cap K[X].$

(c) $\Rightarrow$ (b). Let $g\in W_1\cap K[X]$ be nonzero. If $g \in M_1$, then $g \in M_2 \cap K[X] \subseteq W_2 \cap K[X]$.
Next assume that $g \not\in M_1$ and $g \not\in W_2 \cap K[X]$. Since $W_2$ is a torsion extension of $V$,
there are $n \in \mathbb{N}$ and $b \in V$ such that $w_2(g^n) = nw_2(g) = -v(b) <0${, i.e.,} $w_2(g^nb) =0$.
Note that $w_1(g^nb) = nw_1(g)+w_1(b) = v(b) >0$, so $g^nb \in M_1 \cap K[X] \subseteq M_2 \cap K[X]$
by assumption. Hence $w_2(g^nb) > 0$, a contradiction, which implies $g \in W_2 \cap K[X]$.
Thus $W_1\cap K[X]\subseteq W_2\cap K[X].$

(b) $\Rightarrow$ (a). Assume that $f\in K[X]$ is nonzero. Then either $f \in W_1$ or $\frac{1}{f} \in W_1$, and since
$W_1$ is torsion over $V$, there exist a positive integer  $n\in\N$ and $0 \neq c\in K$ such that $nw_1(f)=v(c)$.
Hence, $w_1(\frac{f^n}{c}) = 0$.
In particular, $\frac{f^n}{c}\in W_1\cap K[X]\subseteq W_2\cap K[X]$, so that we also have
$$0\leq w_2\left(\frac{f^n}{c}\right)\Leftrightarrow w_2(f)\geq\frac{v(c)}{n}=w_1(f),$$
so the inequality holds true.

(2) It is clear that $W_1=W_2$ implies both of $W_1\cap K[X]=W_2\cap K[X]$
and $M_1\cap K[X]= M_2\cap K[X]$. Conversely, if $W_1\cap K[X]=W_2\cap K[X]$ or $M_1\cap K[X]= M_2\cap K[X]$,
then $w_1(f)=w_2(f)$ for every $f\in K[X]$ by (1), which implies that $w_1=w_2$ on $K(X)$. Thus, $W_1=W_2$.

(3) This follows directly from (1) and (2) above.
\end{proof}

Let the notation be as in Lemma \ref{ordering}.
Then $w_1<w_2$ if and only if $W_1\cap K[X]\subset W_2\cap K[X]$,
if and only if $M_1\cap K[X]\subset M_2\cap K[X]$. The next corollary gives an alternative characterization of this strict inequality.

\begin{Cor}\label{w1<w2}
Let $W_i \in \mathcal W_{\t}$ and $M_i$ be the maximal ideal of $W_i$ for $i=1,2$. Then
the following statements are satisfied.
\begin{enumerate}
\item[\upshape (1)] $w_1<w_2$ if and only if $M_1\cap K[X]\subset M_2\cap(W_1\cap K[X])$, and in this case,
$M_1\cap K[X]$ is a prime non-maximal ideal of $W_1\cap K[X]$.
\item[\upshape (2)] $W\in \mathcal{W}_{\t}$ is a maximal element of $(\mathcal W_{\t}, \leq)$ if and only if $W$ is residually algebraic.
\end{enumerate}
\end{Cor}
\begin{proof}
(1) Assume that $w_1<w_2$. Then, by Lemma \ref{ordering}, there exists $f\in (W_2\cap K[X])\setminus (W_1\cap K[X])$, so
$w_1(f)<0\leq w_2(f)$. Since $W_1$ is a torsion extension of $V$,
there exist $n\in\N$ and $0 \neq a\in K$ such that $w_1(f^n)=v(a)$. Then, we have
$$0=w_1\left(\frac{f^n}{a}\right)<w_2\left(\frac{f^n}{a}\right),$$
so $\frac{f^n}{a}\in M_2\cap (W_1\cap K[X])$ but $\frac{f^n}{a}\notin M_1\cap K[X]$. Conversely, suppose that
$M_1\cap K[X]\subset M_2\cap(W_1\cap K[X])$. Then $M_2\cap(W_1\cap K[X]) \subseteq M_2 \cap K[X]$
implies $M_1\cap K[X]\subset M_2 \cap K[X]$. Thus $w_1<w_2$ by Lemma \ref{ordering}.

(2) $(\Rightarrow)$ Suppose that $W$ is residually transcendental. By Theorem \ref{description rt Popescu},
there exists $(\alpha,\delta)\in \overline{K}\times\Gamma_{\overline v}$ such that $W=U_{\alpha,\delta}\cap K(X)$
(where $U$ is a fixed extension of $V$ to $\overline K$). If $\delta'\in\Gamma_v$
is strictly greater than $\delta$, we have
 $W\cap K[X] \subsetneq W' \cap K[X]$, where $W'=U_{\alpha,\delta'}\cap K(X)$. In fact, $U_{\alpha,\delta}\cap \overline K[X]\subsetneq U_{\alpha,\delta'}\cap\overline{K}[X]$ because $\frac{X-\alpha}{c'}$ is in the latter ring but not in the former (where $c'\in K$ has value $\delta'$) and by Theorem \ref{Conjugate Rt ext}, $U_{\alpha,\delta}\cap K[X]\subsetneq U_{\alpha,\delta'}\cap K[X]$.

$(\Leftarrow)$ Suppose that $W\in \mathcal{W}_{\t}$ is residually algebraic. We have the following containments,
 where $M$ is the maximal ideal of $W$:
$$\frac{V}{M \cap K} \subseteq\frac{W\cap K[X]}{M\cap K[X]}\subseteq\frac{W}{M}.$$
It follows that $(W\cap K[X])/(M\cap K[X])$ is a field and thus $M\cap K[X]$ is a maximal ideal of $W\cap K[X]$.
By the first part of the Corollary, $W$ is a maximal element of $(\mathcal W_{\t}, \leq)$.
\end{proof}

\subsection{Residually transcendental extensions over the completion}\label{res transc completion}

Let $\V$ (resp. $\K$) denote the $v$-adic completion of $V$ (resp. $K$),
$\widetilde{\mathcal W}_{\t}$ be the set of all torsion extensions of $\V$  to $\K(X)$,
and $\widetilde{\mathcal W}_{\rt}$ be the set of all residually transcendental
extensions of $\V$ to $\K(X)$. In this subsection,
we show that the map of $\widetilde{\mathcal W}_{\t}$ (resp. $\widetilde{\mathcal W}_{\rt}$)
into $\mathcal W_{\t}$ (resp. $\mathcal W_{\rt}$) given by $W \mapsto W \cap K(X)$ is {injective (resp. bijective)}.

\begin{Lem}\label{contraction valuation}
Let $\wW$ be an extension of $\V$ to $\K(X)$ such that  $\tilde w(X)>v(a)$, for some $a\in K$. If $W=\wW\cap K(X)$, then we have
$$\frac{\wW\cap\K[X]}{\wM\cap\K[X]}=\frac{W\cap K[X]}{M\cap K[X]},$$
where $\wM,M$ are the maximal ideals of $\wW,W$, respectively.
\end{Lem}

\begin{proof}
Clearly, we have the containment $(\supseteq)$. For the other containment,
we need to show that for every $f\in \wW\cap\K[X]$ there exists $g\in W\cap K[X]$ such that $\tilde w(f-g)>0$.
Suppose that $f(X)=\sum\limits_{i= 0}^d \alpha_i X^i$. We choose $g\in K[X]$, $g(X)=\sum\limits_{{i=0}}^d a_i X^i$,
such that $v(\alpha_i-a_i)+iv(a)>0$, for each $i = 0,\ldots,d$. In particular,
\begin{equation}\label{f-g}
\tilde w(f-g)\geq \min\{v(\alpha_i-a_i)+i\tilde w(X)\}{\geq}\min\{v(\alpha_i-a_i)+iv(a)\}>0
\end{equation}
so that $f-g\in\wM$. Note that $g\in W$, so we get the claim.
\end{proof}

Let $\V_{\infty}=\{\psi\in\K(X)\mid \psi(\infty)\in\V\}$, where $\psi(\infty)$ is defined as $\varphi(0)$ with $\varphi(X)=\psi(1/X)$ (see \cite{PerTransc}). Then
the assumption on $\wW$ of Lemma \ref{contraction valuation} is equivalent to {saying} $\wW\not=\V_{\infty}$.
Note that every $W\in\mathcal{W}_{\t}$ satisfies the assumption of Lemma \ref{contraction valuation}.

\begin{Lem}\label{overring polynomial domain}
Let $\widetilde R$ be an integral domain such that $\V[X/a]\subseteq \widetilde R\subseteq\K[X]$
for some nonzero $a\in K$ and let $R=\widetilde R\cap K[X]$. Let $\wW$ be a valuation domain
of $\K(X)$ extending $\V$  and let $W=\wW\cap K(X)$. Then
$R\subset W$ if and only if $\widetilde R\subset \wW$.
\end{Lem}
\begin{proof}
The implication $(\Leftarrow)$ is clear. Conversely, suppose that $R\subset W$ and let $f\in \widetilde R$. In particular, $f\in\K[X]=\K[X/a]$, so $f(X)=\sum\limits_{i=0}^d \alpha_i (X/a)^i$, for some $\alpha_i\in\K$. Let $g(X)=\sum\limits_{i=0}^d a_i (X/a)^i\in K[X]$ be such that $v(\alpha_i-a_i)\geq0$ for each $i=0,\ldots,d$. Then $h:=f-g\in \V[X/a]\subseteq\widetilde{R}$ so that $g=f-h\in \widetilde R\cap K[X]=R$. By assumption, $g\in W$ and note that $h\in \wW$, because $V[X/a]\subset W$ implies that $\V[X/a]\subset \wW$. Hence,  $f=h+g\in \wW$, as claimed.
\end{proof}

The following corollary follows at once.

\begin{Cor}\label{contraction W polynomial}
Let $\wW_i$ be an extension of $\V$ to $\K(X)$ such that $\tilde w_i(X)>v(a)$ for some $a\in K$, for $i=1,2$. Let $W_i=\wW_i\cap K(X)$, for $i=1,2$. Then we have
\begin{equation}\label{contaiments}
W_1\cap K[X]\subseteq W_2\cap K[X]\Leftrightarrow \wW_1\cap \K[X]\subseteq \wW_2\cap \K[X].
\end{equation}
\end{Cor}
\begin{proof}
It is sufficient to notice that the assumption of $
\tilde w_1$ implies that $\V[X/a]\subseteq\wW_1\cap\K[X]$ and then apply Lemma \ref{overring polynomial domain}.
\end{proof}

\begin{Lem}\label{extension rt ext}
Let $W$ be a residually transcendental extension of $V$ to $K(X)$
and $\widehat{K(X)}$ $($resp. $\W)$ be the completion of $K(X)$ $($resp. $W)$ with respect to $W$.
Then the following statements hold true:
\begin{itemize}
\item[\upshape (1)] $\K(X)$ is a subfield of $\widehat{K(X)}$.
\item[\upshape (2)] $\W$ is immediate over $\W\cap\K(X)$.
\item[\upshape (3)] $\W\cap\K(X)$ is a residually transcendental extension of $\V$.
\end{itemize}
\end{Lem}
\begin{proof}
Since $\Gamma_w/\Gamma_v$ is torsion, it follows that $\K$ is equal to the topological closure of $K$ in $\widehat{K(X)}$.
Hence, we may embed $\K(X)$ into $\widehat{K(X)}$.

Let $\widetilde W=\W\cap \K(X)$.
Since $V\subseteq\V$ and $W\subseteq\widehat{W}$ are immediate extensions, it follows that $\widehat{W}$ is immediate over $\widetilde W$ and $\W$ is residually transcendental over $\V$.
In particular, since $K(X)\subseteq\K(X)\subseteq\widehat{K(X)}$, it follows that also $W\subseteq\widetilde{W}$ is
an immediate extension and so $\widetilde W$ is also a residually transcendental extension of $\V$. If $X$ were algebraic over $\K$,
then  $\K\subseteq\K(X)$ would be a finite extension and so $\widetilde W$ could not be residually transcendental over $\V$.
Then $X$ is transcendental over $\K$.
\end{proof}

\begin{Lem}\label{contraction rt ext}
Let $\wW$ be a residually transcendental extension of $\V$ to $\K(X)$. Then $\wW\cap K(X)$ is a residually transcendental extension of $V$.
\end{Lem}
\begin{proof}
By \cite[Proposition 1]{AlePop}, there exists $f\in\K[X]\cap \wW$ such that its residue modulo $\widetilde M$ is transcendental over $\V/\M=V/M$. By Lemma \ref{contraction valuation},
if $g\in K[X]$ is a polynomial of the same degree as $f$ sufficiently $v$-adically close to $f$,
then $g\in \wW\cap K(X)$ and $f$ and $g$ have the same residue modulo $\widetilde M$ so $g$ is residually transcendental over $V/M$.
Thus, $\wW\cap K(X)$ is a residually transcendental extension of $V$.
\end{proof}

\begin{Thm}\label{contraction torsion exts}
If $\wW\in \widetilde{\mathcal W}_{\t}$ then $\wW\cap K(X)\in \mathcal W_{\t}$. Moreover,  if $\wW_i\in \widetilde{\mathcal W}_{\t}$ and $W_i=\wW_i\cap K(X)$ for $i=1,2$, then $w_1\leq w_2$ if and only if $\tilde w_1\leq \tilde w_2$. In particular, $W_1=W_2$ if and only if $\wW_1=\wW_2$.
\end{Thm}
\begin{proof}
It is clear that if $\wW\in\widetilde{\mathcal W}_{\t}$, then $\wW\cap K(X)\in \mathcal W_{\t}$,
because $\Gamma_v=\Gamma_{\hat v}$.
Hence, by Lemma \ref{ordering}, the claim $w_1\leq w_2\Leftrightarrow\tilde w_1\leq \tilde w_2$
is equivalent to \eqref{contaiments}, which holds by Corollary \ref{contraction W polynomial}.

For the final claim, $(\Leftarrow)$ is straightforward. Conversely, suppose $W_1=W_2$.
Then $W_1\cap K[X]= W_2\cap K[X]$, so by \eqref{contaiments}, $\wW_1\cap \K[X]=\wW_2\cap \K[X]$.
Thus, by Lemma \ref{ordering}, $\wW_1=\wW_2$.
\end{proof}

\begin{Cor}\label{bijection rt}
The map $\Phi:\widetilde{\mathcal W}_{\rt} \to \mathcal W_{\rt}$ given by
$\wW \mapsto \wW\cap K(X)$ is bijective.
\end{Cor}
\begin{proof}
Note that if $\wW\in\widetilde{\mathcal W}_{\rt}$, then $W=\wW\cap K(X)\in \mathcal W_{\rt}$ by Lemma \ref{contraction rt ext}.
Hence, $\Phi$ is well-defined. Then $\Phi$ is  surjective by Lemma \ref{extension rt ext}
and $\Phi$ is injective by Theorem \ref{contraction torsion exts}.
\end{proof}

\section{Complete list of DVRs of $\Q(X)$}\label{list DVRs}

In this section we give a complete description of the family of DVRs with quotient field $\Q(X)$.

Recall that $\PP$ denotes the set of prime numbers.
Given a DVR $W$ of $\Q(X)$,
 either $W\cap \Q=\Q$ or $W\cap \Q=\Z_{(p)}$, for some $p\in\PP$. In the first case we say that $W$ is a \emph{non-unitary} valuation domain, and in the second case, that $W$ is a \emph{unitary} valuation domain. For each $p\in\PP$, let $\mathcal{W}_p$ be the set of DVRs $W$ of $\Q(X)$ such that $W \cap \Q = \Z_{(p)}$.

It is a classical result (see for example \cite[p. 3]{Chevalley}) that if $W$ is a non-unitary valuation domain, then $W$ is equal to $\Q[X]_{(g)}$, for some irreducible $g\in\Q[X]$ or to $\Q[1/X]_{(1/X)}$, according to whether $X\in W$ or $X\not\in W$.

\noindent
If $W\in \mathcal{W}_p$ for some $p\in\PP$,
we distinguish between these two subcases:
\begin{enumerate}
\item $W$ is residually algebraic over $\Z_{(p)}$.
\item $W$ is residually transcendental over $\Z_{(p)}$.
\end{enumerate}

\noindent
For each $p\in\PP$ we denote by $\mathcal{W}_{p,\ra}$ the subset of $\mathcal{W}_p$ given by the residually algebraic extensions of $\Z_{(p)}$ and by $\mathcal{W}_{p,\rt}$ the subset of $\mathcal{W}_p$ given by the residually transcendental extensions of $\Z_{(p)}$.
In particular, $\mathcal{W}_{p} = \mathcal{W}_{p,\ra} \cup \mathcal{W}_{p,\rt}$ and  $\mathcal{W}_{p,\ra} \cap \mathcal{W}_{p,\rt}= \emptyset$.

\subsection{Extensions of $\Z_{(p)}$ to $\Q(X)$}

Throughout the paper we {keep the same notation used in the introduction.}
We say that an element $\alpha\in\C_p^{\br}$ is unramified over $\Q_p$ if the ramification index of $\Q_p(\alpha)/\Q_p$ is $1$ (we stress that such an element $\alpha$ may not be algebraic over $\Q_p$). Note that $\oQp\subset\C_p^{\br}\subset\C_p$. We also set $\O_p^{\br}=\O_p\cap\C_p^{\br}$.

 For $(\alpha,\delta)\in\C_p\times(\Q\cup\{\infty\})$, we {denote by $\O_{p,\alpha,\delta}$, $\Z_{p,\alpha,\delta}$ and $\Z_{(p),\alpha,\delta}$ the associated monomial valuation domains of $\C_p(X),\Q_p(X),\Q(X)$, respectively (see \S \ref{monomial valuations}).
 If $\delta=\infty$, we may also use $\Z_{(p),\alpha}$ in place of $\Z_{(p),\alpha,\infty}$ (see Remark \ref{ra rt}).}

\subsubsection{Residually algebraic extensions}

The following theorem characterizes the DVRs of $\Q(X)$ which are residually algebraic extensions of $\Z_{(p)}$, for some $p\in\PP$ (see \cite[Corollary 2.28]{PerStacked}).

\begin{Thm}\label{DVR ra}
Let $W$ be a DVR of $\Q(X)$ which is a residually algebraic extension of $\Z_{(p)}$ for some $p\in\PP$.
Then there exists $\alpha\in\C_p^{\br}$, transcendental over $\Q$, such that $W=\Z_{(p),\alpha}$.
Furthermore,  $\alpha \in \oQp$ if and only if the extension of the residue fields $W/M\supseteq\Z/p\Z$ is finite.
\end{Thm}

\begin{Rem}\label{why alpha transcendental}
We remark that if $\alpha\in\oQp$, we may define in general the following valuation domain of $\Q(X)$:
$$\Z_{(p),\alpha}=\{\phi\in\Q(X)\mid \phi(\alpha)\in \oZp\}$$
By \cite[Proposition 2.2]{PerTransc}, $\Z_{(p),\alpha}$ is a DVR precisely if $\alpha$ is transcendental over $\Q$, it has rank $2$ otherwise. In this last case, if $q\in\Q[X]$ is the minimal polynomial of $\alpha$, we have $\Z_{(p),\alpha}\subset\Q[X]_{(q)}$. In either case, $\Z_{(p),\alpha}$ is always a residually algebraic extension of $\Z_{(p)}$.
\end{Rem}

\subsubsection{Residually transcendental extensions}

 The next theorem gives a characterization of the residually transcendental extensions of $\Z_{(p)}$ to $\Q(X)$, alternative to that of Theorem \ref{description rt Popescu}, by means of pairs in $\oQp\times\Q$ rather than in $\overline{\Q}\times\Q$.

\begin{Thm}\label{rt over completion}
Let $p\in\PP$ and let $W$ be a residually transcendental extension  of $\Z_{(p)}$ to $\Q(X)$.
Then $W$ is a DVR which is  equal to  $\Z_{p,\alpha,\delta}\cap\Q(X)$ for some $(\alpha,\delta)\in\oQp\times\Q$.
\end{Thm}
\begin{proof}
We have already said that $W$ is a DVR in Remark \ref{remark rt}. By Corollary \ref{bijection rt}, $W$ is equal to the restriction to $\Q(X)$ of a (unique) residually transcendental extension $\wW$ of $\Z_p$ to $\Q_p(X)$. By Theorem \ref{description rt Popescu}, $\wW=\Z_{p,\alpha,\delta}$ for some $(\alpha,\delta)\in\oQp\times\Q$. Hence, $W=\Z_{p,\alpha,\delta}\cap\Q(X)$.
\end{proof}

\subsubsection{Unitary DVRs of $\Q(X)$}\label{res tr}

In order to unify the  residually algebraic and transcendental extensions of $\Z_{(p)}$ to $\Q(X)$, we consider  pairs $(\alpha,\delta)\in\C_p\times(\Q\cup\{\infty\})$. In the following, we will always assume the following conditions for such a pair:
\begin{itemize}
    \item[-] if $\delta=\infty$, then $\alpha$ is transcendental over $\Q$.
    \item[-] if $\delta\in\Q$, then $\alpha\in \oQp$.
\end{itemize}
The first assumption is due to Theorem \ref{DVR ra}. The second assumption is due to the following  density argument: given $(\alpha,\delta)\in\C_p\times\Q$, for any $\alpha'\in\bar B_p(\alpha,\delta)$  we have over $\C_p$ the equality $v_{p,\alpha,\delta}=v_{p,\alpha',\delta}$ by \eqref{equality monomials} and Lemma \ref{ordering}(2). Since $\oQp$ is dense in $\C_p$, we can also suppose that $\alpha'$ is in $\oQp$. Explicitly, we may consider  $\alpha'=\alpha+p^kt$, where $k\geq\delta$ is an integer and $t\in\Z_p$ is transcendental over $\Q$.
 By our convention, given a pair $(\alpha,\delta)\in \C_p\times(\Q\cup\{\infty\})$ with $\alpha\not\in\oQp$, automatically $\delta=\infty$. The converse is not true: we can consider pairs $(\alpha,\infty)$, where $\alpha$ is either in $\oQp$ or in $\C_p\setminus\oQp$.

Recall that $\Z_{(p),\alpha}=\Z_{(p),\alpha,\infty}$ by Remark \ref{ra rt}.
The next theorem summarizes the results discussed so far and {gives} a complete description of the unitary DVRs of $\Q(X)$.

\begin{Thm}\label{unitary DVRs over Z(p)}
Let $p\in\PP$.
Then
the following statements are satisfied.
\begin{enumerate}
\item[\upshape (1)] {Let $(\alpha,\delta)\in\C_p^{\br}\times(\Q\cup\{\infty\})$ satisfy the assumptions of \S \ref{res tr}. Then $\Z_{(p),\alpha,\delta}$ is a DVR of $\Q(X)$ lying above $\Z_{(p)}$.}

\item[\upshape (2)] If $W$ is a DVR of $\Q(X)$ lying above $\Z_{(p)}$, then $W=\Z_{(p),\alpha,\delta}$ for some $(\alpha,\delta)\in\C_p^{\br}\times(\Q\cup\{\infty\})$. {Moreover}, $X\in W$ if and only if $(\alpha,\delta)\in\O_p^{\br}\times(\Q_{\geq0}\cup\{\infty\})$.
\item[\upshape (3)] $W$ is residually transcendental (algebraic, respectively) over $\Z_{(p)}$
if and only if $\delta<\infty$ {($\alpha$} is transcendental over $\Q$ and $\delta=\infty$, respectively).
\end{enumerate}
\end{Thm}

\begin{proof}
(1) Let $(\alpha,\delta)\in\C_p^{\br}\times(\Q\cup\{\infty\})$ be as in \S \ref{res tr}. If $\delta\in\Q$, then $\alpha\in\oQp$ and, by Remark \ref{remark rt}(1) applied to $V=\Z_p$, $\Z_{p,\alpha,\delta}$ is a residually transcendental extension of $\Z_p$ to $\Q_p(X)$; hence $\Z_{(p),\alpha,\delta}=\Z_{p,\alpha,\delta}\cap\Q(X)$ is a residually transcendental extension of $\Z_{(p)}$  by Lemma \ref{contraction rt ext}, so it is a DVR by Remark \ref{remark rt}(1). If $\delta=\infty$, then $\alpha$ is transcendental over $\Q$ and $\Z_{(p),\alpha,\infty}=\Z_{(p),\alpha}$ is a DVR by \cite[Proposition 2.2]{PerTransc} if $\alpha\in\oQp$,  and by \cite[Proposition 2.14, Remark 2.18 and Corollary 2.28]{PerStacked} if $\alpha\in\C_p^{\br}\setminus\oQp$.

(2) Let $W$ be a DVR of $\Q(X)$ lying above $\Z_{(p)}$. If $W$ is residually transcendental over $\Z_{(p)}$, then there exists $(\alpha,\delta)\in\oQp\times\Q$
such that $W=\Z_{(p),\alpha,\delta}$ (Theorem \ref{rt over completion}).
If $W$ is residually algebraic over $\Z_{(p)}$, then there exists $\alpha\in\C_p^{\br}$
such that  $\alpha$ is transcendental over $\Q$ and  $W=\Z_{(p),\alpha}$ (Theorem \ref{DVR ra}).
 For the second part of item (2), note that $X\in W=\Z_{p,\alpha,\delta}\cap\Q(X)$ if and only if $X\in\Z_{p,\alpha,\delta}$, that is, $\Z_p[X]\subset\Z_{p,\alpha,\delta}$. If $\delta<\infty$, then $\alpha\in\oQp$ and, by Lemma \ref{rt overring of V[X]} applied to $V=\Z_p$, this holds if and only if $v_p(\alpha)\geq0$ and $\delta\geq0$. If $\delta=\infty$, then $X\in\Z_{(p),\alpha}$ if and only if $\alpha\in\O_p$.

 (3) By the proof of (1), if $\delta<\infty$, then $\Z_{(p),\alpha,\delta}$ is residually transcendental over $\Z_{(p)}$. If $\delta=\infty$, then the map $\phi\mapsto\phi(\alpha)$ induces an embedding of the residue field of $W=\Z_{(p),\alpha}$ into the residue field of $\O_p$, which is algebraic over $\F_p$, so $W$ is residually algebraic over $\Z_{(p)}$; moreover, $\alpha$ is transcendental over $\Q$ by our convention.
\end{proof}

\addtocontents{toc}{\protect\setcounter{tocdepth}{1}}
\subsection{Polynomial ordering of unitary valuation domains and ultrametric balls in $\C_p$}\label{Polordering monomial}

In this subsection, we study under which conditions we have $\Z_{(p),\alpha',\delta'}\cap\Q[X]\subseteq \Z_{(p),\alpha,\delta}$ for
$(\alpha,\delta),(\alpha',\delta')\in\C_p\times(\Q\cup\{\infty\})$, and we show that  $\mathcal{W}_p$ is a partially ordered set under the corresponding ordering.
\addtocontents{toc}{\protect\setcounter{tocdepth}{2}}
\addcontentsline{toc}{subsection}{\protect\tocsubsection{}{\thesubsection}{Polynomial ordering of unitary valuation domains and balls in $\C_p$}}

Now, in order to show the correspondence between the valuation domains of $\mathcal{W}_p$ and ultrametric balls  in $\C_p$, we need to consider ultrametric balls possibly having infinite radius (thus, a singleton). {Given $(\alpha,\delta)\in\C_p\times(\Q\cup\{\infty\})$, the ball of center $\alpha$ and radius $\delta$ is $\bar B_p(\alpha,\delta)=\{x\in\C_p\mid v_p(x-\alpha)\geq\delta\}$, }   so $\bar B_p(\alpha,\delta)=\{\alpha\}$ precisely when $\delta=\infty$. If $\delta\in\Q$,
then $\bar B_p(\alpha,\delta)$ has a non-empty intersection with $\oQp$ by density (by the same  arguments given in \S \ref{res tr}), and so we may find an element in  $\bar B_p(\alpha,\delta)\cap\oQp$ which can be considered as the center of the ball.

We will show that the ordering on $\mathcal{W}_p$ reflects the containment among ultrametric balls in $\C_p$ under the action of the Galois group $G_p$ (Proposition \ref{overring intersection DVRs} and Remarks \ref{ball containment}).
We first need a couple of lemmas. The first lemma uses the notion of pseudo-stationary sequence introduced in \cite{ChabPolCloVal}, see also \cite{PerPrufer}.

\begin{Lem}\label{lem:ball containment}
Let $(\alpha,\delta)\in \C_p\times(\Q\cup\{\infty\})$ and suppose that $\bar B_p(\alpha,\delta)\subseteq \bigcup\limits_{j=1,\ldots,n}B_j$,
where the $B_j$'s are polynomially closed subsets of $\C_p$. Then $\bar B_p(\alpha,\delta)\subseteq B_i$ for some $i\in\{1,\ldots,n\}$.
\end{Lem}
\begin{proof}
If $\delta=\infty$, then $\bar B_p(\alpha,\delta)=\{\alpha\}$ and the claim clearly is true.

Suppose that $\delta\in\Q$. Let $E=\{s_n\}_{n\in\N}\subset \bar B_p(\alpha,\delta)$ be a pseudo-stationary sequence with set of pseudo-limits $\mathcal{L}_E=\bar B_p(\alpha,\delta)$.
By the pigeonhole principle, a subsequence of $E$, say $E'=\{s_{n_k}\}_{k\in\N}$, is contained in $B_i$, for some $i\in\{1,\ldots,n\}$;
note that $E'$ is still pseudo-stationary with $\mathcal{L}_{E'}=\bar B_p(\alpha,\delta)$. By \cite[Proposition 4.8]{ChabPolCloVal},
$\bar B_p(\alpha,\delta)\subseteq B_i$, as claimed.
\end{proof}

We denote by $G_p$ the Galois group of $\oQp/\Q_p$. We recall that $G_p$ is identified with
the group of all continuous $\Q_p$-{automorphisms} of $\C_p$ (see \cite[\S 3, 1.]{APZclosed}) and so $G_p$ acts on $\C_p$.
Moreover, for each $\sigma\in G_p$, we have $v_p\circ\sigma=v_p$, because $\Q_p$ is Henselian.  In particular, if $U$ is the (unique) extension of $\Z_p$ to $\oQp$, then $\mathcal D(U)=G_p$. Hence,  by Theorem \ref{Conjugate Rt ext} applied to $K=\Q_p$, for $(\alpha,\delta)\in\oQp\times\Q$ and $\sigma\in G_p$ we have $\Z_{p,\sigma(\alpha),\delta}=\Z_{p,\alpha,\delta}$, and so $\Z_{(p),\sigma(\alpha),\delta}=\Z_{(p),\alpha,\delta}$. Similarly, by \eqref{inclusion}, \eqref{equality monomials} and Lemma \ref{ordering}(2) applied to $K=\C_p$, if $(\alpha_1,\delta_1),(\alpha_2,\delta_2)\in\C_p\times\Q$ are such that $\delta_1\leq\delta_2$ and $v_p(\alpha_1-\alpha_2)\geq\delta_1$, then $\O_{p,\alpha_1,\delta_1}\cap\C_p[X]\subseteq\O_{p,\alpha_2,\delta_2}\cap\C_p[X]$, so that $\Z_{(p),\alpha_1,\delta_1}\cap\Q[X]\subseteq\Z_{(p),\alpha_2,\delta_2}\cap\Q[X]$ and equality holds if $\delta_1=\delta_2$;
in this last case we also have $\Z_{(p),\alpha_1,\delta_1}=\Z_{(p),\alpha_2,\delta_2}$ by Lemma \ref{ordering}(2).
Finally, if $\alpha\in\C_p$ and $\sigma\in G_p$, then $\Z_{(p),\sigma(\alpha)}=\Z_{(p),\alpha}$, since $\phi(\sigma(\alpha))=\sigma(\phi(\alpha))$ for every $\phi\in\Q(X)$. In the rest of the paper we use these facts without further mention.

\begin{Rem}\label{conjugate transcendental element}
Let $\alpha\in\C_p$. By \cite[Remark 3.2]{APZclosed}, $G_p(\{\alpha\})=\{\sigma(\alpha)\mid \sigma\in G_p\}$ is
an infinite subset of $\C_p$ precisely when $\alpha\notin\oQp$. In this case, by \cite[Theorem 3.5]{APZclosed}, $G_p(\{\alpha\})$
is totally disconnected, closed and compact (note that this is true also when $G_p(\{\alpha\})$ is a finite set).
In particular, by \cite[Proposition 2.4]{ChabPolCloVal}, $G_p(\{\alpha\})$ is polynomially closed.
\end{Rem}

\begin{Lem}\label{union balls pol closed}
Let $(\alpha_i,\delta_i)\in\C_p\times(\Q\cup\{\infty\})$, for $i=1,\ldots,n$. Then, $G_p(\bigcup\limits_{i=1}^n \bar B_p(\alpha_i,\delta_i))$ is polynomially closed.
\end{Lem}
\begin{proof}
We remark that
\begin{equation}\label{conjugates union balls}
G_p(\bigcup_{i=1}^n \bar B_p(\alpha_i,\delta_i))=\bigcup_{i=1}^n G_p(\bar B_p(\alpha_i,\delta_i))=\bigcup_{i=1}^n \bigcup_{\sigma\in G_p}\bar B_p(\sigma(\alpha_i),\delta_i).
\end{equation}

If $\alpha_i\in\oQp$, then $G_p(\bar B_p(\alpha_i,\delta_i))$ is equal to a finite union of balls $\bar B_p(\sigma(\alpha_i),\delta_i)$,
for $\sigma\in G_p$ (which may not be necessarily disjoint);
therefore, by \cite[Proposition 3.4]{ChabPolCloVal}, $G_p(\bar B_p(\alpha_i,\delta_i))$ is polynomially closed.
If $\alpha_i\in\C_p\setminus\oQp$ (so $\delta_i$ is necessarily equal to $\infty$ by our convention),
then $G_p(\bar B_p(\alpha_i,\delta_i))=G_p(\{\alpha_i\})$ is polynomially closed by Remark \ref{conjugate transcendental element}.
Hence, $G_p(\bigcup\limits_{i=1}^n \bar B_p(\alpha_i,\delta_i))$ is polynomially closed, being a finite union of polynomially closed sets.
\end{proof}

We are now ready to prove the main result of this subsection by which we
can show that $\mathcal{W}_p$ is a partially ordered set. We need first the following remark, which gives the connection with rings of integer-valued polynomials.

\begin{Rem}\label{balls and intval rings}
  Let $(\alpha,\delta)\in\C_p\times(\Q\cup\{\infty\})$. If $\delta\not=\infty$, given $c\in\oQp$ such that $v_p(c)=\delta$ we have the following equalities:
$$
\O_{p,\alpha,\delta}\cap\C_p[X]=\O_p\left[\frac{X-\alpha}{c}\right]=\Int(\bar B_p(\alpha,\delta),\O_p).$$
The first equality follows from \cite[Remark 3.4]{PerPrufer} and the second one
from \cite[Lemma 4.1]{PerPrufer}, since $\O_p$ is a rank one non-discrete valuation domain with infinite residue field. More generally, it is clear that the equality
\begin{equation}\label{restriction valuation domain on polynomial ring}
\O_{p,\alpha,\delta}\cap\C_p[X]=\Int(\bar B_p(\alpha,\delta),\O_p)
\end{equation}
is true also when $\delta=\infty$.
In particular, contracting down the previous equality to $\Q_p[X]$ (and similarly to $\Q[X]$), we have
\begin{equation}\label{IntQpB}
\Z_{p,\alpha,\delta}\cap\Q_p[X]=\Int_{\Q_p}(\bar B_p(\alpha,\delta),\O_p)
\end{equation}

\end{Rem}

\begin{Prop}\label{overring intersection DVRs}
Let $p\in\PP$ and $(\alpha,\delta),(\alpha_i,\delta_i)\in\C_p\times(\Q\cup\{\infty\})$, for $i=1,\ldots,n$.
Then the following conditions are equivalent:
\begin{itemize}
\item[\upshape (1)] $\Z_{(p),\alpha_i,\delta_i}\cap\Q[X]\subset\Z_{(p),\alpha,\delta}$ for some $i\in\{1,\ldots,n\}$.
\item[\upshape (2)] $\bigcap\limits_{i=1}^n\Z_{(p),\alpha_i,\delta_i}\cap\Q[X] \subset\Z_{(p),\alpha,\delta}$.
\item[\upshape (3)] $\delta\geq\delta_i$ and $v_p(\alpha-\sigma(\alpha_i))\geq\delta_i$ for some $\sigma\in G_p$ and $i\in\{1,\ldots,n\}$.
\end{itemize}
{In particular, if $n=1$, we have $$\Z_{(p),\alpha_1,\delta_1}\cap\Q[X]\subset\Z_{(p),\alpha,\delta}\Leftrightarrow \bar B_p(\alpha,\delta)\subseteq\sigma(\bar B_p(\alpha_1,\delta_1))=\bar B_p(\sigma(\alpha_1),\delta_1)\ \text{ for some }\sigma\in G_p.$$}
\end{Prop}

\begin{proof}
 (1) $\Rightarrow$ (2) Clear.

 (2) $\Rightarrow$ (3) {By Lemma \ref{overring polynomial domain}, the containment in \textup{(2)}} is equivalent to the analogous containment in $\Q_p(X)$:
\begin{equation}\label{overring of intersection}
\bigcap_{i=1}^n(\Z_{p,\alpha_i,\delta_i}\cap\Q_p[X])\subset\Z_{p,\alpha,\delta}.
\end{equation}
By \eqref{IntQpB}, the left-hand side of \eqref{overring of intersection} is equal to
$$\bigcap_{i=1}^n\Int_{\Q_p}(\bar B_p(\alpha_i,\delta_i),\O_p)=\Int_{\Q_p}(B,\O_p),$$
where $B=\bigcup\limits_{i=1}^n \bar B_p(\alpha_i,\delta_i)$.

A straightforward adaptation of \cite[Lemma 2.20]{PerDedekind} shows that
the integral closure of $\Int_{\Q_p}(B,\O_p)$ in $\oQp(X)$ is equal to $\Int_{\oQp}(G_p(B),\O_p)$. Hence, \eqref{overring of intersection} is actually equivalent to
$$\Int_{\oQp}(G_p(B),\O_p)\subset (\oZp)_{\alpha,\delta}.$$
By applying Lemma \ref{overring polynomial domain} again, we can transfer the previous containment over $\C_p$:
$$\Int(G_p(B),\O_p)\subset\O_{p,\alpha,\delta}\cap\C_p[X]=\Int(\bar B_p(\alpha,\delta),\O_p)$$
{which amounts to saying} that $\bar B_p(\alpha,\delta)$ is contained in the polynomial closure of $G_p(B)$. By Lemma \ref{union balls pol closed}, $G_p(B)$ is polynomially closed, so $\bar B_p(\alpha,\delta)\subseteq G_p(B)$. By Lemma \ref{union balls pol closed} again,
for each $i\in\{1,\ldots,n\}$, $G_p(\bar B_p(\alpha_i,\delta_i))$ is polynomially closed,
so by Lemma \ref{lem:ball containment}, $\bar B_p(\alpha,\delta)\subseteq G_p(\bar B_p(\alpha_i,\delta_i))$ for some $i\in\{1,\ldots,n\}$.
If $\delta_i=\infty$, then necessarily $\delta=\infty$ and $\alpha=\sigma(\alpha_i)$ for some $\sigma\in G_p$,
because otherwise a similar argument to the proof of Lemma \ref{lem:ball containment} shows
that $G_p(\{\alpha_i\})$ would contain a pseudo-stationary sequence which is not possible by compactness.
If $\delta_i<\infty$, again Lemma \ref{lem:ball containment} implies that $\bar B_p(\alpha,\delta)$ is contained
in $\sigma(\bar B_p(\alpha_i,\delta_i))=\bar B_p(\sigma(\alpha_i),\delta_i)$ for some $\sigma\in G_p$,
{which amounts to saying} that $\delta\geq\delta_i$ and $v_p(\alpha-\sigma(\alpha_i))\geq\delta_i$, which is precisely condition (3).

(3) $\Rightarrow$ (1) By Theorem \ref{Conjugate Rt ext},  $\Z_{(p),\alpha_i,\delta_i}=\Z_{(p),\sigma(\alpha_i),\delta_i}$ for each $\sigma\in G_p$. Note also that, by \eqref{equality monomials}, the inequality $v_p(\alpha-\sigma(\alpha_i))\geq\delta_i$
 implies that $\Z_{(p),\sigma(\alpha_i),\delta_i}=\Z_{(p),\alpha,\delta_i}$. Therefore,
$$\Z_{(p),\alpha_i,\delta_i}\cap\Q[X]=\Z_{(p),\alpha,\delta_i}\cap\Q[X]\subset\Z_{(p),\alpha,\delta}$$
where the last containment follows from the inequality $\delta\geq\delta_i$.

{The last claim follows at once: condition (3) is equivalent to the containment of the corresponding ultrametric balls.}
\end{proof}

\begin{Rems}\label{ball containment}
(1) The residually algebraic DVRs $\Z_{(p),\alpha}=\Z_{(p),\alpha,\infty}$ are precisely the maximal elements of $\mathcal{W}_p$
under the ordering defined in \textup{(2)} below (Corollary \ref{w1<w2}(2)); moreover, each residually transcendental extension $\Z_{(p),\alpha,\delta}\in\mathcal{W}_p$, $\delta\in\Q$, is strictly less than a maximal element of $\mathcal{W}_p$, namely, $\Z_{(p),\alpha',\infty}$ for  any  $\alpha'\in\bar B_p(\alpha,\delta)\cap\oQp$ transcendental over $\Q$; such an $\alpha'$ exists by the arguments given in \S \ref{res tr}.

(2) Let $(\alpha_i,\delta_i)\in\C_p\times(\Q\cup\{\infty\})$ for $i=1,2$ be such that $\Z_{(p),\alpha_1,\delta_1}\cap\Q[X] = \Z_{(p),\alpha_2,\delta_2} \cap \Q[X]$. By Proposition \ref{overring intersection DVRs}, we have
$\delta_1 = \delta_2$, $v_p(\alpha_2-\sigma_1(\alpha_1)) \geq \delta_1$, and
$v_p(\alpha_1-\sigma_2(\alpha_2)) \geq \delta_2$ for some $\sigma_1, \sigma_2 \in G_p$.
If $\delta_1 = \infty$, then $\alpha_1 = \sigma(\alpha_2)$ for some $\sigma \in G_p$ (that is, $\alpha_1,\alpha_2$ are conjugate over $\Q_p$),
so, in particular, $\Z_{(p),\alpha_1,\delta_1} = \Z_{(p),\alpha_2,\delta_2}$. If instead we have $\delta_i < \infty$
for $i=1, 2$, then $\Z_{(p),\alpha_1,\delta_1} = \Z_{(p), {\sigma_1}(\alpha_1),\delta_1}$ by Theorem \ref{Conjugate Rt ext}
and $\Z_{(p), \sigma_1(\alpha_1),\delta_1} = \Z_{(p), \alpha_2,\delta_1}$ by (\ref{equality monomials}),
and hence again we get $\Z_{(p),\alpha_1,\delta_1} = \Z_{(p),\alpha_2,\delta_2}$.
Thus, $\mathcal{W}_p$ is a partially ordered set under $\Z_{(p),\alpha_1,\delta_1}\cap\Q[X] \subseteq \Z_{(p),\alpha_2,\delta_2} \cap \Q[X]$
for $\Z_{(p),\alpha_i,\delta_i} \in \mathcal{W}_p$.
\end{Rems}

 Let $\mathcal{B}_p$ be the set of ultrametric balls $\bar B_p(\alpha,\delta)$ in $\C_p$  with center in $\C_p^{\br}$ and possibly infinite radius,
with the conventions of \S\ref{res tr}. Note that the Galois group $G_p=\Gal(\oQp/\Q_p)$ acts on $\mathcal{B}_p$ (Lemma \ref{image rt} and Theorem \ref{Conjugate Rt ext}){;} we denote by $\mathcal{B}_p/G_p$ the set of equivalence classes of $\mathcal{B}_p$ under this action.  Note that, by Remark \ref{conjugate transcendental element}, the cardinality of an equivalence class of a ball $\bar B_p(\alpha,\delta)$ under the action of $G_p$ is infinite precisely when $\delta=\infty$ and $\alpha\not\in\oQp$.
\begin{Thm}\label{ultrametric balls modulo Galois are DVRs}
The map
$$\begin{array}{lcc}
   \Theta_p:&\mathcal{B}_p/G_p\to & \mathcal{W}_p\\
   &{[} \bar{B}_p(\alpha,\delta){]}\mapsto &\Z_{(p),\alpha,\delta}
\end{array}$$
is a bijection.
\end{Thm}
\begin{proof}
   We show first that $\Theta_p$ is well-defined. If $[\bar{B}_p(\alpha,\delta)]=[\bar{B}_p(\alpha',\delta')]$ then there exists $\sigma\in G_p$ such that $\bar{B}_p(\alpha,\delta)=\sigma(\bar{B}_p(\alpha',\delta'))=\bar{B}_p(\sigma(\alpha'),\delta')$, so in particular we get $\delta=\delta'$ and $v_p(\alpha-\sigma(\alpha'))\geq\delta$. Hence, by Proposition \ref{overring intersection DVRs}, we have $\Z_{(p),\alpha,\delta}\cap\Q[X]=\Z_{(p),\alpha',\delta'}\cap\Q[X]$ and thus by {Remark \ref{ball containment}(2)} we get $\Z_{(p),\alpha,\delta}=\Z_{(p),\alpha',\delta'}$.

  For the injectivity of $\Theta_p$, if the equality $\Z_{(p),\alpha,\delta}=\Z_{(p),\alpha',\delta'}$ holds, then in particular $\Z_{(p),\alpha,\delta}\cap\Q[X]=\Z_{(p),\alpha',\delta'}\cap\Q[X]$ and the same arguments as above show that $[\bar{B}_p(\alpha,\delta)]=[\bar{B}_p(\alpha',\delta')]$.

   The surjectivity of $\Theta_p$ follows from Theorem \ref{unitary DVRs over Z(p)}.
\end{proof}

\section{{Krull domains between $\Z_{(p)}[X]$ and $\Q[X]$}}\label{section local case}

{In this section we fix a prime $p$ and describe the Krull domains between $\Z_{(p)}[X]$ and $\Q[X]$: they are the rings of integer-valued polynomials on finite unions of ultrametric balls of $\C_p$ (Theorem \ref{thm4.23}). We then determine their prime ideals and divisor class groups, and characterize those which are Dedekind domains or UFDs; together these results prove Theorem \ref{thmA}.}

Given a prime ideal $Q$ of {a subring $R$ of $\Q[X]$ containing $\Z[X]$}, we say that $Q$ is \emph{unitary} if $Q\cap\Z=p\Z$ for some $p\in\PP$ and that $Q$ is \emph{non-unitary} if $Q\cap\Z=(0)$.

{Throughout this section, $p\in\PP$ is fixed and $R$ is a Krull domain between $\Z_{(p)}[X]$ and $\Q[X]$.} Clearly, for every such $R$ we have
\begin{equation}\label{representation R}
R=\bigcap_{j=1}^m W_j\cap\Q[X]
\end{equation}
where $\{{W_j}\mid j=1,\ldots,m\}$ is a finite family (possibly empty) of DVRs of $\Q(X)$  lying above $\Z_{(p)}$. {By Theorem \ref{unitary DVRs over Z(p)}, $W_j=\Z_{(p),\alpha_j,\delta_j}$ for some $(\alpha_j,\delta_j)\in\O_p^{\br}\times (\Q_{\geq0}\cup\{\infty\})$, for $j=1,\ldots,m$ (note that $X\in W_j$, since $\Z_{(p)}[X]\subseteq R$). We recall that  $\Z_{(p)}[X]$ is a Krull domain (for example see \cite[Proposition 1.6]{Fossum}) and $\Z_{(p)}[X]=\Z_{(p),0,0}\cap\Q[X]=\IntQ(\O_p)$.}

\subsection{{Representation as rings of integer-valued polynomials}}\label{subsec:local representation}

We first investigate when a representation of a Krull domain between $\Z_{(p)}[X]$ and $\Q[X]$ as in \eqref{representation R} is irredundant. {We use the notions of superfluous members and irredundant representations recalled in \S\ref{subsec:Krull}.}

\begin{Lem}\label{center superfluous DVR}
Let $R$ be as in \eqref{representation R} and let $W=W_j$ for some $j\in\{1,\ldots,m\}$ , with maximal ideal $M$. If $W'$ is a DVR extending $\Z_{(p)}$, with maximal ideal $M'$, such that $w<w'$, then $R\subset W'$ and the center of $W'$ on $R$ has height two. Moreover, if $W''$ is another DVR extending $\Z_{(p)}$, with maximal ideal $M''$, such that $w<w'\leq w''$, then $M\cap R\subset M'\cap R=M''\cap R$.
\end{Lem}

\begin{proof}
By assumption that $w<w'$, we have $R\subseteq W\cap\Q[X]\subseteq W'\cap\Q[X]$.
Hence $W'$ is superfluous in $R \cap W'$, and since $W \neq W'$, we have $\operatorname{ht}(M' \cap R) \geq 2$ by Corollary \ref{cor10}. Note that $\dim(R) \leq 2$.
Thus $\operatorname{ht}(M' \cap R) = 2$.

For the last statement, we have $M\cap R\subseteq M'\cap R$ by Lemma \ref{ordering}(1), and $R\subset W'\cap\Q[X]\subseteq W''\cap\Q[X]$, so $M'\cap R\subseteq M''\cap R$. Since $\dim(R) \leq 2$, equality holds in the previous containment.
\end{proof}

\begin{Prop}\label{nonunitary not superfluous}
Let $R$ be as in \eqref{representation R}. Then, for every $q\in\Q[X]$ irreducible, $\Q[X]_{(q)}$ is not superfluous in that representation of $R$.
\end{Prop}
\begin{proof}
Let $Q$ be the center of $\Q[X]_{(q)}$ on $R$.
Then $Q \cap \Z = (0)$, and hence if $S = \Z \setminus \{0\}$, then $R_S = \bigcap\limits_{{j=1}}^{{m}}({W_j})_S \cap \Q[X] = \Q[X]$
\cite[Proposition 43.5]{Gilm} and $QR_S \subsetneq R_S$. Thus {$\operatorname{ht}(Q)=\operatorname{ht}(QR_S) = 1$},
so, by Corollary \ref{cor10}{(2)}, $\Q[X]_{(q)}$ is not superfluous.
\end{proof}

\begin{Rem}\label{irredundant represent balls}
(1) By Proposition \ref{overring intersection DVRs},  geometrically the representation of $R$ in \eqref{representation R} is irredundant if and only if the balls $\bar B_p(\alpha_j,\delta_j)$, $(\alpha_j,\delta_j)\in\O_p^{\br}\times(\Q_{\geq0}\cup\{\infty\})$, for $j=1,\ldots,{m}$, are disjoint under the action of the Galois group $G_p$ in the following sense: for each pair of balls $\bar B_p(\alpha_{j_1},\delta_{j_1}),\bar B_p(\alpha_{j_2},\delta_{j_2})$, there is no $\sigma\in G_p$ such that $\bar B_p(\alpha_{j_1},\delta_{j_1})$ is contained in $\sigma(\bar B_p(\alpha_{j_2},\delta_{j_2}))$ (see Remarks \ref{ball containment}). Equivalently, for $j_1\neq j_2$ and $\sigma_1,\sigma_2\in G_p$, the balls $\bar B_p(\sigma_1(\alpha_{j_1}),\delta_{j_1})$ and $\bar B_p(\sigma_2(\alpha_{j_2}),\delta_{j_2})$ are disjoint. We stress that the number of such balls is finite if and only if there is no $i\in\{1,\ldots,{m}\}$ such that $\alpha_i\in\O_p^{\br}\setminus\oZp$ (see Remark \ref{conjugate transcendental element}).

(2) By Corollary \ref{cor10} and Proposition \ref{nonunitary not superfluous},
the representation of $R$ in \eqref{representation R} is irredundant
if and only if $M_j\cap R$ has height one for $j=1,\ldots,m$, where $M_j$ is the maximal ideal of $W_j$.
\end{Rem}

We can now characterize Krull {domains} $R$ between $\Z_{(p)}[X]$ and $\Q[X]$ {in a suitable way as rings of integer-valued polynomials}.

\begin{Thm} \label{thm4.23}
Let $p\in\PP$ and  $R$ be a Krull domain  between $\Z_{(p)}[X]$ and $\Q[X]$. Then there exist finitely many pairs $(\alpha_j,\delta_j)\in\O_p^{\br}\times(\Q_{\geq0}\cup\{\infty\})$,  $j=1,\ldots,m$, with  $\alpha_j$ transcendental over $\Q$ if $\delta_j=\infty$, such that
\begin{equation}\label{representation Krull local}
R=\bigcap_{j=1}^m\Z_{(p),\alpha_j,\delta_j}\cap\Q[X]=\IntQ\left(E_p,\O_p\right),
\end{equation}
where $E_p = \bigcup\limits_{j=1}^m \bar B_p(\alpha_j,\delta_j)$. Moreover,
if the representation in \eqref{representation Krull local} is irredundant (see Remark \ref{irredundant represent balls}),
then $R$ is Dedekind if and only if either $m=0$ (and then $R=\Q[X]$) or $m\geq1$ and  $\delta_j=\infty$ for each $j=1,\ldots,m$;  thus, in either case, $R$ is Dedekind if and only if $E_p$ contains finitely many elements.
\end{Thm}
\begin{proof}
{Everything follows by \eqref{representation R} and Theorem \ref{unitary DVRs over Z(p)} as we said above.}
The representation of $R$ as a ring of integer-valued polynomials follows by \eqref{IntQpB}.

Clearly, if the family of unitary DVRs $W_j$ is empty, then $R=\Q[X]$, and in this case $R$ is a PID (thus, Dedekind and a UFD). {Otherwise}, by \cite[Theorem 3.14]{PerPrufer}, the Krull domain $R$ is Dedekind if and only if no $W_j$ is residually transcendental over $\Z_{(p)}$. This corresponds to saying that $\delta_j=\infty$ for each $j=1,\ldots,m$, by Theorem \ref{unitary DVRs over Z(p)}.
\end{proof}

\subsection{Centers of monomial valuations on $\Z_{(p)}[X]$}\label{centers}

{Let  $M_{(p),\alpha,\delta}$ be the maximal ideal of $\Z_{(p),\alpha,\delta}$.} In this subsection, we restrict to the {study of the valuation overrings $\Z_{(p),\alpha,\delta}$ of $\Z_{(p)}[X]$, that is,  $(\alpha,\delta)\in\O_p^{\br}\times(\Q_{\geq0}\cup\{\infty\})$ (Theorem \ref{unitary DVRs over Z(p)})}.  We describe the center of $\Z_{(p),\alpha,\delta}$ on $\Z_{(p)}[X]$, namely, $M_{(p),\alpha,\delta}\cap (\Z_{(p)}[X])$.

For each $p \in \PP$, let $\mathbb{F}_p$ be the finite field of $p$ elements, and $\overline{f}$ be the image of $f \in \Z_{(p)}[X]$ in $\mathbb{F}_p[X]$.
Let $\operatorname{Spec}(\Z[X])$ (resp., $\operatorname{Max}(\Z[X])$) be the set of prime ideals (resp., maximal ideals) of $\Z[X]$.
It is well known and easy to see that
\begin{enumerate}
\item[\upshape (a)] $X^1(\Z[X]) = \{p\Z[X] \mid p \in \PP\} \cup \{f\Z[X] \mid f\in\Z[X]$ and $f$ is irreducible in $\Z[X]\}$
\item[\upshape (b)] $\operatorname{Max}(\Z[X]) = \{(p, f) \mid p \in \PP, f\in\Z[X]$ and $\overline{f}\F_p[X]$ is a prime ideal$\}$.
\end{enumerate}
In particular, for each $p \in \PP$,
\begin{eqnarray*}
\operatorname{Spec}(\Z_{(p)}[X]) &=& \{0\} \cup  \{p\Z_{(p)}[X]\}\\
                         &\cup& \{f\Z_{(p)}[X] \mid f\in\Z_{(p)}[X] \text{ and }f \text{ is {irreducible} in  }\Z_{(p)}[X]\} \\
&\cup&\{(p, f) \mid \overline{f}\F_p[X] \text{ is a prime ideal }\}.
\end{eqnarray*}

{For $\alpha \in \mathbb{O}_p$, we denote by $\overline{\alpha}$ the image of $\alpha$ modulo the maximal ideal $\mathbb{M}_p$ of $\O_p$ and by $g_{\alpha}(X) \in \mathbb{Z}_{(p)}[X]$ a monic polynomial such that} $\overline{g_{\alpha}}(X)$ is the minimal polynomial {of} $\overline{\alpha}$ over $\mathbb{F}_p$.

In the proof of the next proposition, the case $0<\delta<\infty$ is reduced to the case $\delta=\infty$ by means of Lemma \ref{center superfluous DVR}.

\begin{Prop}\label{contraction maximal ideal}
Let $p\in\PP$ and $(\alpha,\delta)\in\O_p\times(\Q_{\geq0}\cup\{\infty\})$. Then
$$\Z_{(p)}[X]\cap M_{(p),\alpha,\delta}=\left\{\begin{array}{lc}
(p),&\text{ if } \delta=0\\
(p,g_{\alpha}(X)),&\text{ if } \delta>0
\end{array}\right.$$
\end{Prop}
\begin{proof} Case 1: $\delta=0$.  Then $\Z_{(p),\alpha,0}=\Z_{(p),0,0}= \Z_{(p)}[X]_{p\Z_{(p)}[X]}$.
In particular, $\Z_{(p)}[X]=\Z_{(p),0,0}\cap\Q[X]$, so, the center of $M_{(p),\alpha,0}=M_{(p),0,0}$ on $\Z_{(p)}[X]$ is equal to $p\Z_{(p)}[X]$.

Case 2: $\delta=\infty$. Since $(p,g_{\alpha}(X))$  is a maximal ideal of $\Z_{(p)}[X]$ and $g_{\alpha}\in M_{(p),\alpha,\infty}$,  $(p,g_{\alpha}(X))$ is the center of $\Z_{(p),\alpha,\infty}$ on $\Z_{(p)}[X]$.

Case 3: $0<\delta<\infty$. In this case, $\alpha\in\oQp$ by our convention (\S\ref{res tr}).  By the same subsection, we may also assume that $\alpha$ is transcendental over $\Q$, 
this changes neither $\Z_{(p),\alpha,\delta}$ nor $\overline{\alpha}$, hence nor $g_\alpha$. 
In particular,
$\Z_{(p),\alpha,\infty}$ is a DVR by Theorem \ref{unitary DVRs over Z(p)}(1). Let $W=\Z_{(p),\alpha,0}$, $W'=\Z_{(p),\alpha,\delta}$ and $W''=\Z_{(p),\alpha,\infty}$. Since $\bar B_p(\alpha,0)=\O_p=\bar B_p(0,0)$, we have $W=\Z_{(p),0,0}$, so $R=\Z_{(p)}[X]=W\cap\Q[X]$ is as in \eqref{representation R} with $m=1$ and $W_1=W$. Since $\{\alpha\}\subsetneq\bar B_p(\alpha,\delta)\subsetneq\bar B_p(\alpha,0)$, by Proposition \ref{overring intersection DVRs} we have $W\cap\Q[X]\subseteq W'\cap\Q[X]\subseteq W''\cap\Q[X]$, and these three DVRs are distinct by Theorem \ref{ultrametric balls modulo Galois are DVRs}; hence $w<w'<w''$ by Lemma \ref{ordering}. Therefore, by Lemma \ref{center superfluous DVR}, we get
$$R\cap M_{(p),\alpha,\delta}=R\cap M_{(p),\alpha,\infty}=(p,g_{\alpha}(X)),$$
where the second equality follows from Case 2.
\end{proof}

\begin{Cor}\label{equality centers on Z[X]}
Let $p\in\PP$ and
$(\alpha_i,\delta_i)\in\O_p\times(\Q_ {\geq0}\cup\{\infty\})$, for $i=1,2$.\\ Then $\Z_{(p)}[X]\cap M_{(p),\alpha_1,\delta_1}=\Z_{(p)}[X]\cap M_{(p),\alpha_2,\delta_2}$ if and only if one of the following conditions holds:
\begin{enumerate}
\item[\upshape (a)]  $\delta_1=\delta_2=0$;
\item[\upshape (b)] $\delta_i\not=0$, for $i=1,2$,  and $\overline{\alpha_1},\overline{\alpha_2}$  are  conjugate over $\F_p$.
Equivalently, $v_p(\sigma(\alpha_1)-\alpha_2)>0$ for some $\sigma\in G_p=\Gal(\oQp/\Q_p)$.
\end{enumerate}

\end{Cor}
\begin{proof}
$(\Rightarrow)$ Suppose first that $\Z_{(p),\alpha_1,\delta_1},{\Z_{(p),\alpha_2,\delta_2}}$ have the same center on $\Z_{(p)}[X]$.
 By Proposition \ref{contraction maximal ideal}, we have either $\delta_1=\delta_2=0$ or $\delta_i\not=0$ for $i=1,2$. In the last case, we have $\Z_{(p)}[X]\cap M_{(p),\alpha_i,\delta_i}=(p,g_{\alpha_i}(X))$, where $g_{\alpha_i}\in\Z_{(p)}[X]$ is a monic polynomial such that $\overline{g_{\alpha_i}}$ is the minimal polynomial of $\overline{\alpha_i}$ over $\F_p$, for $i=1,2$.  Then, these prime ideals are the same if and only if $\overline{g_{\alpha_1}}=\overline{g_{\alpha_2}}$, that is, $\overline{\alpha_1},\overline{\alpha_2}$ are conjugate over $\F_p$. Note that this last condition amounts
 to saying that there exists some $\overline{\sigma}\in \Gal(\overline{\F_p}/\F_p)$ such that $\overline{\sigma}(\overline{\alpha_1})=\overline{\alpha_2}$. Since $\overline{\sigma}(\overline{\alpha_1})=\overline{\sigma(\alpha_1)}$ for some $\sigma\in G_p$  \cite[Chapt. V, \S{}2.2, Proposition 6(ii)]{Bourb}), it follows that $v_p(\sigma(\alpha_1)-\alpha_2)>0$, as claimed.

$(\Leftarrow)$ If $\delta_1=\delta_2=0$, then $\Z_{(p)}[X]\cap M_{(p),\alpha_i,\delta_i}=(p)$ for $i=1,2$.
If instead $\delta_i\not=0$ for $i=1,2$ and $\overline{\alpha_1},\overline{\alpha_2}$ are conjugate over $\F_p$,
then by the above arguments we have $\Z_{(p)}[X]\cap M_{(p),\alpha_1,\delta_1}=\Z_{(p)}[X]\cap M_{(p),\alpha_2,\delta_2}$.
\end{proof}

In Corollary \ref{equality centers on Z[X]}, if $\delta_i\not=0$ for $i=1,2$, then $\Z_{(p),\alpha_1,\delta_1}$ and $\Z_{(p),\alpha_2,\delta_2}$ have the same center on $\Z_{(p)}[X]$ if and only if $\overline{\alpha_1},\overline{\alpha_2}$ are conjugate over $\F_p$, which is much weaker than $\Z_{(p),\alpha_1,\delta_1}=\Z_{(p),\alpha_2,\delta_2}$, that is, $\delta_1=\delta_2$ and $v_p(\sigma(\alpha_1)-\alpha_2)\geq\delta_1$ for some $\sigma\in G_p$ (Theorem \ref{ultrametric balls modulo Galois are DVRs}). For example, if $\alpha_1,\alpha_2\in \O_p$ are such that $v_p(\alpha_1-\alpha_2)>0$ (that is, they belong to the same ball contained in $\O_p$ of strictly positive radius), then $\overline{\alpha_1}=\overline{\alpha_2}$, so $\Z_{(p),\alpha_1,\delta_1}$ and $\Z_{(p),\alpha_2,\delta_2}$ have the same center on $\Z_{(p)}[X]$ for all $\delta_1,\delta_2>0$.

\subsection{{Prime ideals, ramification and divisor class group}}\label{subsec:local primes}

Now we study some basic ring-theoretic properties of {a polynomial Krull domain $R$ between $\Z_{(p)}[X]$ and $\Q[X]$. }
In particular, we characterize when the unitary prime ideals of $R$ are maximal. If $W$ is a DVR of $\Q(X)$ extending $\Z_{(p)}$ for some $p\in\PP$, we denote by $e(W\mid \Z_{(p)})$ the ramification index of the extension $W\supset\Z_{(p)}$. If $\alpha\in\C_p^{\br}$, we denote by $e(\Q_p(\alpha)\mid\Q_p)$ the ramification index of the extension $\Q_p(\alpha)\supseteq\Q_p$ (we stress that this extension is algebraic precisely when $\alpha\in \oQp$).

\begin{Thm} \label{characterization local Krull domain}
Let $p\in\PP$. Suppose that $W_j=\Z_{(p),\alpha_j,\delta_j}$, for some $(\alpha_j,\delta_j)\in\C_p^{\br}\times(\Q\cup\{\infty\})$,
is a DVR of $\Q(X)$ extending $\Z_{(p)}$ and $X \in W_j$ for $j = 1, \dots , m$. Now, let
 $M_j=M_{(p),\alpha_j,\delta_j}$ for each $j=1,\ldots,m$, and
\begin{equation}\label{representation R2}
R = \bigcap_{j=1}^m W_j\cap\Q[X].
\end{equation}
\noindent
Then $R$ is a Krull domain such that $\Z_{(p)}[X] \subseteq R \subseteq \Q[X]$. If moreover the representation of $R$ in \eqref{representation R2} is irredundant (namely, $\operatorname{ht}(M_j \cap R)=1$ for $j=1, \dots , m$), then statements {\textup{(1)-- (4)} } also hold:
\begin{enumerate}
\item[\upshape (1)] $X^1(R) = \{M_j \cap R \mid j =1, \dots , m\} \cup \{f\Q[X] \cap R \mid f \in \Q[X]$ is irreducible$\}$.
\item[\upshape (2)] $M_{j} \cap R$ is a maximal ideal of $R$ if and only if $W_j$
is residually algebraic over $\Z_{(p)}$.
\item[\upshape (3)] Suppose that $(\alpha_j,\delta_j)$ is a minimal pair if $\delta_j\in\Q$. Then we have
 $$e(W_j|\Z_{(p)}) = \left\{\begin{array}{cc}
e(\Q_p(\alpha_j)|\Q_p),&\text{ if } \delta_j = \infty\\
e(\Q_p(\alpha_j)|\Q_p)e(\gamma_j|\Q_p(\alpha_j)) ,&\text{ if } \delta_j < \infty
\end{array}\right.$$
where $\gamma_j=w_j(f_{\alpha_j})$, $f_{\alpha_j}$ is the minimal polynomial of $\alpha_j$ over $\Q_p$ and $e(\gamma_j|\Q_p(\alpha_j))$
is the smallest natural number $n$ such that $n\gamma_j$ is in the value group of $\Q_p(\alpha_j)$.

\item[\upshape (4)] If  $m\geq1$,  $e_j = e(W_j|\Z_{(p)})$ for $j=1, \dots , m$ and $d= \gcd(e_1, \dots , e_m)$, then
$\operatorname{Cl}(R) = \Z/d\Z \oplus \Z^{m-1}$.

\end{enumerate}
\end{Thm}
\begin{proof}
{Since} $R$ is a finite intersection of Krull domains with the same quotient field, {it follows that $R$ is a Krull domain }(\cite[Chapter VII, \S 1, 3. Exemples, 3)]{Bourb}).

(1) The containment $(\supseteq)$ follows by \cite[Corollary 2.3(3)]{Chang}
and the assumption that $M_j\cap R\in X^1(R)$ for all $j\in\{1,\ldots,m\}$.
The other containment follows easily by  Theorem  \ref{thm9}.

(2) If $W_j$ is residually transcendental over $\Z_{(p)}$, that is, $\delta_j\in\Q$ (Theorem \ref{unitary DVRs over Z(p)}),
then without loss of generality we  may suppose that $\alpha_{j}$ is  transcendental over $\Q$
(as in the proof of Proposition \ref{contraction maximal ideal}).
Then $W=\Z_{(p),\alpha_j,\infty}$ is a DVR (Theorem \ref{unitary DVRs over Z(p)}(1))
and $w_j<w$ (Proposition \ref{overring intersection DVRs}, Theorem \ref{ultrametric balls modulo Galois are DVRs}
and Lemma \ref{ordering}). By Lemma \ref{center superfluous DVR}, $M_j\cap R\subseteq M\cap R$ and $\operatorname{ht}(M\cap R)=2$, where $M$ is the maximal ideal of $W$; since $\operatorname{ht}(M_j\cap R)=1$, $M_j\cap R$ is not a maximal ideal of $R$. Suppose instead that $W_j$ is residually algebraic over $\Z_{(p)}$, that is, {$\delta_j=\infty$} (Theorem \ref{unitary DVRs over Z(p)}).  We have the following containments: $\Z/p\Z\subseteq R/(M_j\cap R)\subseteq W_j/M_j$; since $W_j/M_j$ is algebraic over $\Z/p\Z$, it follows that $M_j\cap R$ is a maximal ideal of $R$.

(3) If $\delta_j<\infty$, see \cite[Theorem 2.1, c)]{APZTheorem}, while if $\delta_j=\infty$, see \cite[Proposition 2.2]{PerTransc}  if $\alpha_j\in\oQp$ and \cite[Propositions 2.14 and 2.24]{PerStacked} if $\alpha_j\in\C_p^{\br}\setminus\oQp$.

(4) Let $G$ be a free abelian group of rank $m$ on the generators $g_1, \dots, g_m$
and $H$ be the subgroup of $G$ generated by $r = \sum\limits_{i=1}^me_ig_i$. Then $\operatorname{Cl}(R) = G/H$ \cite[Theorem 2.4]{Chang},
and since $d= \gcd(e_1, \dots , e_m)$, $\operatorname{Cl}(R) = \Z/d\Z \oplus \Z^{m-1}$
\cite[Remark 1.6]{Chang}.
\end{proof}

\medskip
Let $R$ be a polynomial Krull domain between $\Z_{(p)}[X]$ and $\Q[X]$ as in \eqref{representation R}.
Let $q\in\Q[X]$ be irreducible and let $Q=q(X)\Q[X]_{(q)}\cap R$. By {Theorem \ref{characterization local Krull domain}(1)}, $Q$ is {a height-one} prime ideal of $R$ and $R_Q=\Q[X]_{(q)}$.  Next, we characterize when $Q$ is a maximal ideal of $R$. We will use Proposition \ref{subvaluation of DVR} and the remark following it.

\begin{Thm}\label{nonunitary prime ideals}
  Let $p\in\PP$ and let $R=\bigcap\limits_{j=1}^m \Z_{(p),\alpha_j,\delta_j}\cap \Q[X]$ be a Krull domain, where $(\alpha_j,\delta_j)\in \C_p\times (\Q\cup\{\infty\})$ for $j=1,\ldots,m$.  Let $Q=q(X)\Q[X]\cap R$ be a non-unitary prime ideal of $R$, for some irreducible $q\in \Q[X]$. Then $Q$ is a maximal ideal of $R$ if and only if $q$ has no root in $\bigcup\limits_{j=1}^m \bar B_p(\alpha_j,\delta_j)$.
\end{Thm}
\begin{proof}
Note that when $R$ is Dedekind, then $Q$ is automatically maximal, since $\dim(R)=1$. If moreover the above representation of $R$ is irredundant, then $\delta_j=\infty$ for $j=1,\ldots,m$ by Theorem \ref{thm4.23}, and each ball $\bar B_p(\alpha_j,\delta_j)$ is the singleton $\{\alpha_j\}$ which is necessarily transcendental over $\Q$ because otherwise $\Z_{(p),\alpha_j,\infty}$ would not be a DVR.

In general, suppose that $q(\alpha')=0$ for some $\alpha'\in\bigcup_{j=1}^m \bar B_p(\alpha_j,\delta_j)$. Then, we have these containments:
$$R\subset\Z_{(p),\alpha'}=\Z_{(p),\alpha',\infty}\subset\Q[X]_{(q)}$$
where the first containment follows by Proposition \ref{overring intersection DVRs} and the second one from the fact that $\Q[X]_{(q)}=\{\phi\in {\Q(X)}\mid \phi(\alpha')\in\oQp\}$. In particular, $Q\subset M_{(p),\alpha'}\cap R$, so $Q$ is not a maximal ideal.

Conversely, suppose that $Q\subset M\subset R$ for some height $2$ maximal ideal $M$ of $R$, which necessarily contains $p$.  By \cite[Corollary 19.7]{Gilm}, there exists a $2$-dimensional valuation overring $W$ of $R$ with non-zero prime ideals $Q'\subset M'$ such that $Q'\cap R=Q$ and $M'\cap R=M$. Note that $W\cap \Q=\Z_{(p)}$ and $W_{Q'}=R_Q=\Q[X]_{(q)}$. By Proposition \ref{subvaluation of DVR}, there exist an extension $U$ of $\Z_{(p)}$ to $\Q[X]/(q)$  and a root $\alpha'$ of $q(X)$ such that $W=\{\phi\in\Q(X)\mid \phi(\alpha')\in U\}$. By the remark after that Proposition, we may assume that $\alpha'\in\oQp$, necessarily algebraic over $\Q$, so that $W=\Z_{(p),\alpha'}=\Z_{(p),\alpha',\infty}$. Since $R\subset \Z_{(p),\alpha'}$, it follows by Proposition \ref{overring intersection DVRs} that $\alpha'$ (or possibly some of its {conjugates}) is in $\bigcup_{j=1}^m \bar B_p(\alpha_j,\delta_j)$, as claimed.
\end{proof}

\begin{Cor}
Let $R$ and $Q$ be as in Theorem \ref{nonunitary prime ideals} and suppose that $Q$ is not a maximal ideal. Then the maximal ideals of $R$ containing $Q$ are unitary and they are finitely many.
\end{Cor}
\begin{proof}
    If $Q\subset M\subset R$ where $M$ is a maximal ideal, the above proof of Theorem \ref{nonunitary prime ideals} shows that $M=M_{(p),\alpha'}\cap R$, where $\alpha'$ is a root of $q(X)$. Hence, such $M$'s are finitely many.
 \end{proof}

\begin{Ex}\label{ex:local}
{Let $p=2$ and $R=\IntQ(E_2,\O_2)=\Z_{(2),0,1/2}\cap\Z_{(2),1,1}\cap\Q[X]$, where $E_2=\bar B_2(0,1/2)\cup\bar B_2(1,1)$. The two balls are disjoint and their centers lie in $\Q_2$, so the representation is irredundant (Remark \ref{irredundant represent balls}). Both DVRs of this representation  are residually transcendental over $\Z_{(2)}$, so $R$ is not Dedekind and, by Theorem \ref{characterization local Krull domain}(2), their centers $P_1$, $P_2$ on $R$ are height-one prime ideals which are not maximal; by Proposition \ref{contraction maximal ideal}, $P_1$ lies over $(2,X)$ and $P_2$ over $(2,X+1)$ in $\Z_{(2)}[X]$. The pairs $(0,1/2)$ and $(1,1)$ are minimal pairs, and Theorem \ref{characterization local Krull domain}(3) gives the ramification indices $e_1=2$ (since $v_{2,0,1/2}(X)=1/2$) and $e_2=1$, so $\operatorname{Cl}(R)\cong\Z/\gcd(2,1)\Z\oplus\Z\cong\Z$ by Theorem \ref{characterization local Krull domain}(4). Finally, by Theorem \ref{nonunitary prime ideals}, $X\Q[X]\cap R$ is not maximal, since $0\in E_2$, whereas $(X^2+X+1)\Q[X]\cap R$ is maximal: the roots of $X^2+X+1$ have residues in $\mathbb F_4\setminus\mathbb F_2$, so they lie in neither ball.}
\end{Ex}

\subsection{Unique factorization domains}\label{section UFD local}

In this {subsection} we completely describe the UFDs between $\Z_{(p)}[X]$ and $\Q[X]$.
 It is known that an integral domain is a UFD if and only if it is a Krull domain with trivial divisor class group \cite[Proposition 6.1]{Fossum}.

{The following lemma follows from \cite[Propositions 3.2 and 3.7]{PS2}.}

\begin{Lem}\label{conjugates in the ball}
Suppose $\alpha\in\oQp$ is unramified over $\Q_p$ and let $q_{\alpha}\in\Q_p[X]$ be its minimal polynomial. Then, $v_{p,\alpha,\delta}(q_{\alpha})=k\delta+\gamma$, for some $\gamma\in\Z$ and $k=\#(\Omega_{\alpha}\cap\bar B_p(\alpha,\delta))$, where $\Omega_{\alpha}$ is the set of conjugates of $\alpha$ over $\Q_p$.
\end{Lem}

Next, we show that the cardinality of the set of conjugates of $\alpha$ over $\Q_p$ in $\bar B_p(\alpha,\delta)$,  namely, the quantity $\#(\Omega_{\alpha}\cap\bar B_p(\alpha,\delta))$ is 1  when $\alpha$ is unramified over $\Q_p$ and $(\alpha,\delta)$ is a minimal pair.

\begin{Prop}\label{unramified implies k=1}
Let $(\alpha,\delta)\in\oQp\times\Q$ be a minimal pair such that $\alpha$ is unramified over $\Q_p$. Then $\#(\Omega_{\alpha}\cap\bar B_p(\alpha,\delta))=1$. In particular, $\gamma=v_{p,\alpha,\delta}(q_{\alpha})=\delta+\lambda$, for some $\lambda\in\Z$ and we therefore have:
$$e(\gamma\mid\Q_p(\alpha))=e(\delta\mid \Q_p)$$
\end{Prop}

\begin{proof}
Suppose first that $\alpha\in\Q_p$, so that $q_{\alpha}(X)=X-\alpha$ and $\Omega_{\alpha}=\{\alpha\}$. Then $v_{p,\alpha,\delta}(q_{\alpha})=\delta$ and the statement is true in this case.

Suppose now that $\alpha\not\in\Q_p$. We consider the following constants:
\begin{eqnarray*}
\delta_{\Q_p}(\alpha)&=&\sup\{v_p(\alpha-c)\mid c\in\oQp, \ [\Q_p(c):\Q_p]<[\Q_p(\alpha):\Q_p]\}\\
\omega_{\Q_p}(\alpha)&=&\sup\{v_p(\alpha-\alpha')\mid\alpha'\in\Omega_{\alpha}\setminus\{\alpha\}\}
\end{eqnarray*}
which are both well-defined because $\alpha\not\in\Q_p$. It is known that $\delta_{\Q_p}(\alpha)$ is a maximum (see for example \cite[Lemma 1.2]{PerStacked}) and so in general we have the inequality $\delta_{\Q_p}(\alpha)\leq\omega_{\Q_p}(\alpha)$ (which follows by the classical Krasner's lemma). Since $\Q_p(\alpha)/\Q_p$ is unramified, by \cite[Theorem 1.3]{Khanduja} we have the equality $\delta_{\Q_p}(\alpha)=\omega_{\Q_p}(\alpha)$. Moreover, from the definitions it follows that $(\alpha,\delta)$ is a minimal pair if and only if $\delta>{\delta_{\Q_p}(\alpha)}$.  Now, if $\beta\in\bar B_p(\alpha,\delta)$, we have
$$v_p(\alpha-\beta)\geq\delta>{\delta_{\Q_p}(\alpha)}=\omega_{\Q_p}(\alpha)$$
so, the only conjugate of $\alpha$ in $\bar B_p(\alpha,\delta)$ is $\alpha$ itself, that is, $\#(\Omega_{\alpha}\cap \bar B_p(\alpha,\delta))=1$.

The last claim now follows by Lemma \ref{conjugates in the ball} and the displayed equation is clear from the fact that $\lambda\in\Z$.
\end{proof}

The next theorem characterizes when a DVR of $\Q(X)$ extending $\Z_{(p)}$ is unramified over $\Z_{(p)}$.

\begin{Thm}\label{unramified DVR over Zp}
Let $W$ be a DVR of $\Q(X)$ extending $\Z_{(p)}$ containing $X$. Then $W$ is unramified over $\Z_{(p)}$ if and only if there exists a minimal pair $(\alpha,\delta)\in\O_p^{\br}\times(\Z_{\geq0}\cup\{\infty\})$ such that $W= \Z_{(p), \alpha, \delta}$ and  $\Q_p(\alpha)/\Q_p$ is unramified.
\end{Thm}
\begin{proof}
By Theorem \ref{unitary DVRs over Z(p)}, there exists $(\alpha,\delta)\in\C_p^{\br}\times(\Q\cup\{\infty\})$ such that $W=\Z_{(p),\alpha,\delta}$.

Suppose first that $W$ is residually algebraic over $\Z_{(p)}$, so, by the same Theorem $\delta=\infty$ and $\alpha$ is transcendental over $\Q$. In this case, if $\alpha\in\oQp$, then by \cite[Proposition 2.2]{PerTransc} we have $e(W\mid\Z_{(p)})=e(\Q_p(\alpha)\mid\Q_p)$. If instead $\alpha\in\C_p\setminus\oQp$, then  $e(W\mid\Z_{(p)})=e(\Z_{p,\alpha,\delta}\mid\Z_p)=e(\Q_p(\alpha)\mid\Q_p)$, where the first equality follows from  \cite[Proposition 2.24]{PerStacked} and the second equality from \cite[Proposition 2.14]{PerStacked}. In either case, the claim follows.

Suppose now that $W$ is residually transcendental over $\Z_{(p)}$, so in particular, $\delta\in\Q$ by the aforementioned Theorem,  and we may assume that $\alpha\in\oQp$. By \cite[Theorem 2.1]{APZTheorem}, $e(W\mid\Z_{(p)})=1$ if and only if the following condition holds:
\begin{itemize}
\item[] there exists a minimal pair $(\alpha,\delta)\in\oQp\times\Q$ of definition of $W$ such that $\Q_p(\alpha)/\Q_p$ is unramified and $w(q_{\alpha})\in\Gamma_{\Q_p(\alpha)}$ where $q_{\alpha}$ is the minimal polynomial of $\alpha$ over $\Q_p$.
\end{itemize}
By Lemma \ref{conjugates in the ball}, we have $w(q_{\alpha})=v_{p,\alpha,\delta}(q_{\alpha})=k\delta+\lambda$, where $k=\#(\Omega_{\alpha}\cap \bar B_p(\alpha,\delta))$ and $\lambda\in \Gamma_{\Q_p(\alpha)}=\Gamma_{\Q_p}=\Z$; in particular, we have $w(q_{\alpha})\in\Gamma_{\Q_p(\alpha)}\Leftrightarrow k\delta\in\Z$.

If $\alpha\in\Q_p$, then $q_{\alpha}(X)=X-\alpha$, so $v_{p,\alpha,\delta}(q_{\alpha})=\delta$ which is therefore an integer.
Suppose that $\alpha\not\in\Q_p$. Since $(\alpha,\delta)$ is  a minimal pair, $\delta>0$. By Proposition \ref{unramified implies k=1} we have $k=1$, so that {the} above displayed condition is equivalent to the condition of the main statement of the theorem.
\end{proof}

\begin{Thm}\label{intersection DVR is UFD}
Let $R$ be a ring between $\Z_{(p)}[X]$ and $\Q[X]$. Then $R$ is a UFD if and only if either $R=\Q[X]$ or $R=W\cap\Q[X]$ where $W$ is a DVR of $\Q(X)$ extending $\Z_{(p)}$ unramified over $\Z_{(p)}$. More precisely, the last condition holds if and only if  there exists $(\alpha,\delta)\in\O_p^{\br}\times(\Z_{\geq0}\cup\{\infty\})${, which is a minimal pair if $\delta<\infty$,} such that $\alpha$ is unramified over $\Q_p$ and in this case $R=\IntQ(\bar B_p(\alpha,\delta),\O_p)$.
\end{Thm}

\begin{proof}
If $R$ is a UFD, then by Theorem \ref{thm4.23} we have $R=\bigcap\limits_{i=1}^m\Z_{(p),\alpha_i,\delta_i}\cap\Q[X]$, for some $(\alpha_i,\delta_i)\in\O_p^{\br}\times(\Q_{\geq0}\cup\{\infty\})$, for $i=1,\ldots,{m}$, and we assume that this representation is irredundant. By items (3) and (4) of Theorem \ref{characterization local Krull domain}, either $R=\Q[X]$ or $R=\Z_{(p),\alpha,\delta}\cap\Q[X]$ and $\Z_{(p),\alpha,\delta}$ is unramified over {$\Z_{(p)}$}. Conversely, if $R$ has the stated form,  then by Theorem \ref{characterization local Krull domain}, $R$ is a Krull domain with trivial class group, hence, a UFD.

The proof of the last claim follows from Theorem \ref{unramified DVR over Zp}.
\end{proof}

\section{{Krull domains between $\Z[X]$ and $\Q[X]$}}\label{section global case}

In this section, we globalize the results of Section \ref{section local case} and characterize the polynomial Krull domains: every such domain is an intersection of $\Q[X]$ with finitely many DVRs lying over $\Z_{(p)}$ for each $p$ in a set $\Lambda$ of primes (\S\ref{subsec:global representations}), and we determine when such an intersection is a Krull domain (Theorem \ref{Krull domains finite character}).
{We recall (see the introduction) that } the polynomial Krull domains are precisely the integrally closed Noetherian domains between $\Z[X]$ and $\Q[X]$.

\subsection{Representations of polynomial Krull domains}\label{subsec:global representations}

We recall {that the PID $\Q[X]$ has the following representation as an intersection of DVRs: $\Q[X]=\bigcap_{q}\Q[X]_{(q)}$, where the intersection is over the set of irreducible polynomials of $\Q[X]$. }

In general, let $R$ be a polynomial Krull domain,
$\Lambda = \{p \in \PP \mid {pR \neq R}\}$,  and $X^1_p(R) = \{Q \in X^1(R) \mid p \in Q\}$ for each $p \in \Lambda$.
Then $X^1_p(R)$ is finite {and $R_Q$ is a DVR for all $Q \in X^1_p(R)$. }
Hence, for each $p \in \Lambda$, there exists a finite set $\{W_{p,1}, \dots , W_{p,m_p}\}$
of DVRs with quotient field $\Q(X)$, so that $W_{p,j} \cap \Q = \Z_{(p)}$ for $j =1, \dots , m_p$, and
\begin{equation}\label{representation Krull domain}
R = \bigcap_{p \in \Lambda}(\bigcap_{j=1}^{m_p}W_{p,j}) \cap \Q[X] = \bigcap_{p \in \Lambda}(\bigcap_{j=1}^{m_p}(W_{p,j} \cap \Q[X])).
\end{equation}
For each $p \in \Lambda$, let
\begin{equation}\label{localization R at p}
R_p = \bigcap_{j=1}^{m_p}(W_{p,j} \cap \Q[X])
\end{equation}
then $R_p$ is a Krull domain between $\Z_{(p)}[X]$ and $\Q[X]$ and $R = \bigcap\limits_{p \in \Lambda}R_p$.

\begin{Lem} \label{lemma43}
Let $p\in\PP$ and $R$ be a ring between $\Z[X]$ and $\Q[X]$ such that
$R = \bigcap\limits_{Q \in X^1(R)}R_Q$ and $R_Q$ is a DVR for all $Q \in X^1(R)$, $X^1_p(R) = \{Q \in X^1(R) \mid p \in Q\}$
and $R_p = (\bigcap\limits_{Q \in X^1_p(R)}R_Q) \cap \Q[X]$.
Then the following statements hold true.
\begin{enumerate}
\item[\upshape (1)] $R = \bigcap\limits_{p \in \PP}R_p$.
\item[\upshape (2)] If $R_{(p)}=(\Z\setminus p\Z)^{-1}R$, then $R_{(p)}=R_p$.
\item[\upshape (3)] $R_{(p)} \cap \Q = \Z_{(p)}$.
\end{enumerate}
\end{Lem}

\begin{proof}
(1) If $Q$ is a prime ideal of $R$, either $Q \cap \Z = (0)$ or $Q \cap \Z = p\Z$ for some $p \in \PP$.
Moreover, if $Q \cap \Z = (0)$, then $\Q[X] \subseteq R_Q$.
Hence, $R = R \cap \Q[X] = \bigcap\limits_{p \in \PP}((\bigcap\limits_{Q \in X^1_p(R)}R_Q) \cap \Q[X])
= \bigcap\limits_{p \in \PP}R_p$.

(2) $(\subseteq)$ Let $\frac{f}{n}\in R_{(p)}$, where $f\in R$ and $n\in\Z$ is not divisible by $p$. Then clearly $\frac{f}{n} \in R_Q$,
for every $Q \in X^1_p(R)$ since the value of each of the corresponding valuation of $n$ is zero and $R\subseteq R_p$.
Thus, $R_{(p)} \subseteq R_p$.
$(\supseteq)$ Conversely, let $h\in R_p$. Then we may write $h=\frac{g}{mp^k}$,
for some $g\in\Z[X]$, $m\in\Z$ not divisible by $p$ and $k\geq 0$.
Then $mh=\frac{g}{p^k}$ is clearly in $R$ (because for each $q\in\PP$, $q\not=p$, $p^k$ is a unit of $R_Q$ for all $Q \in X^1_q(R)$).
Thus, we also have $R_{(p)} \supseteq R_p$.

(3) The result follows from the following two observations: (i) $R_Q \cap \Q = \Z_{(p)}$ for all $Q \in X^1_p(R)$
and (ii) $R_{(p)} = R_p$ by (2).
\end{proof}

Let $R$ be as in Lemma \ref{lemma43} and $R_p' = \bigcap\limits_{Q \in X^1_p(R)}R_Q$ for all $p \in \PP$. Then $R$ is a Krull domain
if and only if $R_p$ is a Krull domain for all $p \in \PP$ and, for any $0 \neq f \in \Z[X]$,
$fR_p' = R_p'$ for all but {finitely} many $p$'s in $\PP$.
{By Section \ref{section local case}, each Krull domain $R_p$ is described by Theorem \ref{thm4.23}; hence it remains to study the above finite character condition.}

\subsection{Finite character}\label{subsec:finite character}

Every polynomial Krull domain is an intersection $R$ as in \eqref{representation Krull domain}. In general,  a ring of this form need not be a polynomial Krull domain, see \cite[Theorem 3.5]{Chang} or the following example.

\begin{Ex} \label{ex430}
For each prime $p \in \PP$, let $\delta_p\in\Q$, $\delta_p>0$, and let
$$R = \bigcap_{p\in\PP}\Z_{(p),0,\delta_p}\cap\Q[X].$$
Then $R$ is not Krull, because $X$ is contained in each center of
the residually transcendental extensions $\Z_{(p),0,\delta_p}$ of $\Z_{(p)}$, for $p\in\PP$.
\end{Ex}

\noindent
In this subsection, we determine when an intersection as in \eqref{representation Krull domain} is a Krull domain. We first need the following definition, which is a variant of the one given in \cite{PerDedekind,PerStacked}
in the context of polynomial Dedekind domains.

\begin{Def}\label{def:polfact}
Let $\Lambda\subseteq\PP$ and $\uE=\prod_{p\in\Lambda}E_p$, where $E_p\subseteq\O_p$ is a (possibly empty) subset for each $p\in\Lambda$. For a nonzero $g\in\Z[X]$, we set
$$\PP_{g,\uE}=\{p\in\Lambda\mid v_p(g(\alpha_p))>0 \text{ for some } \alpha_p\in E_p\}.$$
We say that $\uE$ is \emph{polynomially factorizable} if $\PP_{g,\uE}$ is finite for every nonzero $g\in\Z[X]$.
\end{Def}

If $\Lambda=\PP$, each $E_p$ is nonempty and every element of $E_p$ is transcendental over $\Q$ for every $p\in\PP$, then, by \cite[Lemma 2.12]{PerDedekind}, this is equivalent to the definition given in \cite{PerDedekind,PerStacked}: for each nonzero $g\in\Z[X]$ and $\alpha=(\alpha_p)\in\uE$, there exist $n,d\in\Z$, $n,d\geq1$, such that $g(\alpha)^n/d$ is a unit of $\O$, that is, $v_p(g(\alpha_p)^n/d)=0$ for all $p\in\PP$. Note that
$$\PP_{g,\uE}=\{p\in\Lambda\mid g\in\Z[X]\cap M_{(p),\alpha_p} \text{ for some } \alpha_p\in E_p\},$$
where $M_{(p),\alpha_p}=M_{(p),\alpha_p,\infty}$ is the maximal ideal of $\Z_{(p),\alpha_p}=\Z_{(p),\alpha_p,\infty}$.

We consider the case of a subset $E=\prod\limits_{p\in\PP}E_p$ of $\O$ such that, for each $p\in\PP$, $E_p$ is a finite union of balls in $\O_p^{\br}$:
$$E_p=\bigcup_{j=1}^{m_p}\bar B_p(\alpha_{p,j},\delta_{p,j}),\;\;(\alpha_{p,j},\delta_{p,j})\in\O_p^{\br}\times(\Q_{\geq0}\cup\{\infty\}) \;\text{ for } j=1,\ldots,m_p,$$
with the assumption that $\alpha_{p,j}$ is transcendental over $\Q$ for each $p\in\PP$ and $j\in\{1,\ldots,m_p\}$. If $\delta_{p,j}\not=\infty$ this can always be assumed (see \S \ref{res tr}), 
and if $\delta_{p,j}=\infty$ this condition has to be imposed. For each $p\in\PP$, we consider the set of centers of the balls in $E_p$ of strictly positive radius   given by {$C(E_p)=\{\alpha_{p,j}\mid 1\leq j\leq m_p,\ \delta_{p,j}>0\}$} and set
$$C(\uE)=\prod_{p\in\PP}C(E_p)$$
Note that, for each $p\in\PP$, the set of the centers of $\{\Z_{(p),\alpha_{p,j},\delta_{p,j}}\mid j=1,\ldots,m_p\}$ on $\Z[X]$, namely $\{\Z[X]\cap M_{(p),\alpha_{p,j},\delta_{p,j}}\mid j=1,\ldots,m_p\}$,
is finite {(by Corollary \ref{equality centers on Z[X]}, its cardinality may be strictly smaller than $m_p$)}.

We remark that the choice of a center of an ultrametric ball $\bar B_p(\alpha,\delta)\subset \O_p$ is completely arbitrary (every element of the ball is a center of the ball itself). However, for our purposes, since $\Z_{(p),\alpha,\delta}=\Z_{(p),\alpha',\delta}$ for every $\alpha'\in\bar B_p(\alpha,\delta)$ by Theorem \ref{Conjugate Rt ext}, and since, for $\delta>0$, the center of $\Z_{(p),\alpha,\delta}$ on $\Z_{(p)}[X]$ only depends on the residue $\overline{\alpha}$ by Proposition \ref{contraction maximal ideal} (see also Corollary \ref{equality centers on Z[X]}),
 a possible different choice does not affect the next Theorem \ref{Krull domains finite character}, \textbf{provided} the center of each  ball is transcendental over $\Q$
  (as required in Theorem \ref{Krull domains finite character}); as noted above, this can always be assumed if $\delta_{p,j}\in\Q_{\geq0}$.

\begin{Thm}\label{Krull domains finite character}
 Let $\Lambda \subseteq \PP$ be a subset,  $E=\prod\limits_{p\in\Lambda}E_p\subset\O$ where for each $p\in\Lambda$, $E_p=\bigcup\limits_{j=1}^{m_p}\bar B_p(\alpha_{p,j},\delta_{p,j})$ with $(\alpha_{p,j},\delta_{p,j})\in\O_p^{\br}\times(\Q_{\geq0}\cup\{\infty\})$, $\alpha_{p,j}$ transcendental over $\Q$, for each
$j\in\{1,\ldots,m_p\}$ and let $R=\IntQ(\uE,\O)$, where $\uE=\prod\limits_{p\in\Lambda}E_p$, namely:
$$R=\bigcap_{p\in\Lambda}\bigcap_{j=1}^{m_p}\Z_{(p),\alpha_{p,j},\delta_{p,j}}\cap\Q[X].$$
{Assume that the above representation of $R$ is irredundant. Then} the following conditions are equivalent:
\begin{itemize}
    \item[\upshape (1)] $R$ is a Krull domain.
    \item[\upshape (2)] For each (nonconstant and irreducible) $g\in\Z[X]$, $g$ is contained in finitely many {height-one} unitary prime ideals of $R$.
   \item[\upshape (3)]  For each {nonzero} $g\in\Z[X]$, the following set is finite:
$$\mathbb D_{g,\uE}=\{p\in\Lambda \mid\exists j\in\{1,\ldots,m_p\} \text{ such that }g\in\Z[X]\cap\ M_{(p),\alpha_{p,j},\delta_{p,j}}\}.$$
\item[\upshape (4)] {$C(\uE)=\prod\limits_{p\in\Lambda}C(E_p)$} is polynomially factorizable.
\end{itemize}
 \end{Thm}
We remark that if for some $p\in\Lambda$ we have {$\delta_{p,j}=0$ for some $j\in\{1,\ldots,m_p\}$}, then $m_p=1$ by the assumption that the above  representation of $R$ is irredundant. For such a prime $p$, we have $R_{(p)}=\Z_{(p),\alpha_p,0}\cap\Q[X]=\IntQ(\O_p)=\Z_{(p)}[X]$  (see for example \eqref{IntQpB}) {and $C(E_p)=\emptyset$}.
\begin{proof}
(1) $\Leftrightarrow$ (2) Note that $R$ is a Krull domain if and only if, for every $f\in R$, $f\not=0$,
$f$ is contained in finitely many centers of the defining family of DVR overrings of $R$,
namely, $\Z_{(p),\alpha_{p,j},\delta_{p,j}}$, for $p\in\Lambda$ and $j=1,\ldots,m_p$
and $\Q[X]_{(q)}$, for $q\in\Q[X]$ irreducible.

 Since every $f\in R$ can be written as $\frac{g}{d}$, for some $g\in\Z[X]$
 and $d\in\Z$, $d\not=0$, it is clear that $f$ is contained in finitely many centers of the defining family of DVR overrings of $R$ if and only if the same holds for $g$,
 since the integer $d$ is divisible by finitely many $p\in\Lambda$ and for each $p\in\Lambda$ there are only finitely many {DVR overrings} of $R$ extending $\Z_{(p)}$.
 Moreover, $g$ is clearly contained in finitely many nonunitary prime ideals of $R$ (whose height is one by an argument similar
 to the one of  Proposition  \ref{nonunitary not superfluous}), so $R$ is Krull if and only if condition (2) holds (by assumption and  by Corollary \ref{cor10}, the centers on $R$ of the DVRs $\Z_{(p),\alpha_{p,j},\delta_{p,j}}$ are of height one).

(2) $\Leftrightarrow$ (3)
 Since for each $p\in\Lambda$, we only have finitely many unitary DVR overrings of $R$ extending $\Z_{(p)}$
 (namely, $\Z_{(p),\alpha_{p,j},\delta_{p,j}}$, for $j=1,\ldots,m_p$) and their centers on $R$ are of height one by Corollary \ref{cor10}, it follows that (2) is equivalent to (3).

(3) $\Leftrightarrow$ (4)
 Let $\PP_0$ be the set of primes in $\Lambda$ for which there exists $j\in\{1,\ldots,m_p\}$ such that $\delta_{p,j}=0$. If $p \in \PP_0$, then $m_p =1$
 by {the remarks before this proof, }
 and {so } $M_{(p),\alpha_{p,1},\delta_{p,1}}\cap\Z_{(p)}[X]=(p)$ by Proposition \ref{contraction maximal ideal}. Hence, for each nonzero $g\in\Z[X]$, $\mathbb D_{g,\uE}\cap\PP_0$ is a finite set. If $p\in\Lambda \setminus\PP_0$, {then $\delta_{p,j}>0$ for each $j=1,\ldots,m_p$ and by the same aforementioned Proposition }  we have $M_{(p),\alpha_{p,j},\delta_{p,j}}\cap(\Z_{(p)}[X])=M_{(p),\alpha_{p,j},\infty}\cap(\Z_{(p)}[X])$ {and this is a height-two maximal ideal of $\Z_{(p)}[X]$ for each $j=1,\ldots,m_p$}, so $\mathbb D_{g,\uE}\setminus\PP_0=\PP_{g,C(\uE)}$. Therefore, $\mathbb D_{g,\uE}$ is finite if and only if $\PP_{g,C(\uE)}$ is finite. Since, by Definition \ref{def:polfact}, $C(\uE)$ is polynomially factorizable if and only if $\PP_{g,C(\uE)}$ is finite for every nonzero $g\in\Z[X]$, (3) and (4) are equivalent.
\end{proof}

\subsection{{The classification}}\label{subsec:classification}

Let $\Lambda$ and $R$ be as in Theorem \ref{Krull domains finite character}. {Note that} $\Lambda = \emptyset$ if and only if $R = \Q[X]$, and, if $\Lambda$ is finite, then $R$ is always a Krull domain.
The next theorem is the global version of Theorem \ref{thm4.23} and describes the class of polynomial Krull {domains, which, as noted at the beginning of this section, is precisely the class of integrally closed Noetherian domains between $\Z[X]$ and $\Q[X]$.}

\begin{Thm}\label{main thm polynomial krull domains}
Let $R$ be a polynomial Krull domain. Then
\begin{equation}\label{Krull domain as int ring}
R=\bigcap_{p\in\Lambda}\bigcap_{j=1}^{m_p}\Z_{(p),\alpha_{p,j},\delta_{p,j}}\cap\Q[X]=\IntQ(\uE,\O)
\end{equation}
for some subset $\Lambda\subseteq\PP$, where $(\alpha_{p,j},\delta_{p,j})\in\O_p^{\br}\times(\Q_{\geq0}\cup\{\infty\})$
with $\alpha_{p,j}$ transcendental over $\Q$ if $\delta_{p,j}=\infty$ for each $p\in\Lambda$ and $j=1,\ldots,m_p$;
$\uE=\prod\limits_{p\in\Lambda} E_p$ with  $E_p=\bigcup\limits_{j=1}^{m_p}\bar B_p(\alpha_{p,j},\delta_{p,j})$; and, choosing the centers $\alpha_{p,j}$ of the balls of positive radius transcendental over $\Q$ (which is always possible, see \S\ref{subsec:finite character}),  {$C(\uE)=\prod\limits_{p\in\Lambda}C(E_p)$, with $C(E_p)$ as in \S\ref{subsec:finite character},} is polynomially factorizable.

Moreover, if {the representation in \eqref{Krull domain as int ring} is irredundant (i.e., $\operatorname{ht}(M_{(p),\alpha_{p,j},\delta_{p,j}} \cap R) =1$ for all $p \in \Lambda$ and $j=1, \dots, m_p$)}, then $$\operatorname{Cl}(R) = \bigoplus_{p \in \Lambda}(\Z/d_p\Z \oplus \Z^{m_p-1}),$$ where
$e_{p,j} = e(\Z_{(p),\alpha_{p,j},\delta_{p,j}}|\Z_{(p)})$ for $j=1, \dots , m_p$ and $d_p= \gcd(e_{p,1}, \dots , e_{p,m_p})$
for all $p \in \Lambda$.

{Conversely, let $R$ be as in \eqref{Krull domain as int ring}, where the representation is irredundant, every $\alpha_{p,j}$ is transcendental over $\Q$ and $C(\uE)$ is polynomially factorizable. Then $R$ is a polynomial Krull domain.}
\end{Thm}
\begin{proof}
As we remarked in \eqref{representation Krull domain}, we have
$$R = \bigcap_{p \in \Lambda}(\bigcap_{j=1}^{m_p}W_{p,j}) \cap \Q[X],$$
for some subset $\Lambda\subseteq\PP$ where $W_{p,j}$ is a DVR of $\Q(X)$ lying over $\Z_{(p)}$ for each $j=1,\ldots,m_p$ and $p\in{\Lambda}$.

By  Theorem \ref{unitary DVRs over Z(p)}, for each $p\in\Lambda$, there exist  $(\alpha_{p,j},\delta_{p,j})\in\O_p^{\br}\times(\Q_{\geq0}\cup\{\infty\})$ such that $W_{p,j}=\Z_{(p),\alpha_{p,j},\delta_{p,j}}$, for each $j=1,\ldots,m_p$. In particular, $R$ has the stated form.

Finally, {the defining family of $R$ is irredundant (Theorem \ref{thm9}) and,  since $\Z_{(p),\alpha,\delta}=\Z_{(p),\alpha',\delta}$ for every $\alpha'\in\bar B_p(\alpha,\delta)$ by Theorem \ref{Conjugate Rt ext} (see the discussion before Theorem \ref{Krull domains finite character}), we may choose the centers $\alpha_{p,j}$ of the balls of positive radius transcendental over $\Q$; hence $C(\uE)$} is polynomially factorizable by Theorem \ref{Krull domains finite character}
and $\operatorname{Cl}(R) = \bigoplus\limits_{p \in \Lambda}(\Z/d_p\Z \oplus \Z^{m_p-1})$ by the proof of \cite[Theorem 2.4(4)]{Chang}. {The converse is the implication (4) $\Rightarrow$ (1) of Theorem \ref{Krull domains finite character}.}
\end{proof}

In the next corollary, we completely describe UFDs between $\Z[X]$ and $\Q[X]$.
\begin{Cor}\label{global UFD}
Let $R$ be a UFD between $\Z[X]$ and $\Q[X]$. Then there exist a subset
$\Lambda\subseteq\PP$ and pairs $(\alpha_p,\delta_p)\in\O_p^{\br}\times(\Z_{\geq0}\cup\{\infty\})$ such that
\begin{equation}\label{UFD int val}
R=\bigcap_{p\in\Lambda}\Z_{(p),\alpha_p,\delta_p}\cap\Q[X]=\bigcap_{p\in\Lambda}\IntQ(\bar B_p(\alpha_p,\delta_p),\O_p)
\end{equation}
where, for each $p\in\Lambda$, $\alpha_p$ is unramified over $\Q_p${, $(\alpha_p,\delta_p)$ is a minimal pair if $\delta_p<\infty$,} and $\alpha_p$ is transcendental over
$\Q$ if $\delta_p=\infty$ and the set
$C(\uE)=\prod\limits_{p\in\Lambda}\{\alpha_{p} \mid \delta_{p}>0\}$ is
polynomially factorizable.
\end{Cor}
\begin{proof}
By Theorem \ref{main thm polynomial krull domains}, $R$ is as in \eqref{Krull domain as int ring} (and so, in particular, it is a Noetherian domain by \cite{Heinzer}). By the characterization of the class group of $R$ in the aforementioned Theorem and also Theorem \ref{intersection DVR is UFD}, $R$ is as in the statement of the corollary. Note that the second equality of \eqref{UFD int val} follows from Remark \ref{balls and intval rings}.
\end{proof}

\begin{Ex}\label{ex:global UFD}
{For each prime $p$, let $a_p=\lfloor\log p\rfloor$ and $E_p=\bar B_p(a_p,1)=a_p+p\O_p$, and let $R=\IntQ(\uE,\O)=\bigcap_{p\in\PP}\Z_{(p)}[(X-a_p)/p]$. For a nonzero $g\in\Z[X]$ and $x\in E_p$ we have $g(x)\equiv g(a_p)\pmod{\mathbb M_p}$, and $g(\lfloor\log p\rfloor)$ is a nonzero integer of absolute value less than $p$ for all sufficiently large $p$; hence condition $(\ast)$ of Theorem \ref{thmB} holds, and $R$ is an integrally closed Noetherian domain (after replacing each $a_p$ by a center of the same ball which is transcendental over $\Q$, which does not change $\Z_{(p),a_p,1}$ by Theorem \ref{Conjugate Rt ext}). Each $\Z_{(p),a_p,1}$ is unramified over $\Z_{(p)}$, so $\operatorname{Cl}(R)=0$ and $R$ is a UFD. Note that $R$ contains $(X-a_p)/p$ for every $p$, so it is not finitely generated over $\Z[X]$. In contrast, with $a_p=0$ for all $p$ condition $(\ast)$ fails and the ring is not Noetherian (Example \ref{ex430}).}
\end{Ex}

\subsection{When a polynomial Krull domain is Dedekind}\label{sec:dedekind global}

 We now consider the case in which all the unitary DVRs are residually algebraic: this case gives exactly the polynomial Dedekind domains, and every polynomial Krull domain is the intersection of such a domain with one defined by residually transcendental DVRs (Remark \ref{Krull intersection}).

\begin{Prop}\label{Krull Dedekind}
Let $R$ be a Krull domain between $\Z[X]$ and $\Q[X]$. Then the following statements
are equivalent.
\begin{enumerate}
\item[\upshape (1)] $R$ is a Dedekind domain.
\item[\upshape (2)] Each unitary DVR $W$ of the defining family of $R$ is a residually algebraic extension of $W\cap\Q$.
\item[\upshape (3)] $R = \IntQ(\uE,\O)$, for some subset $\uE=\prod\limits_{p\in\PP}E_p\subset\O=\prod\limits_{p\in\PP}\O_p$ such that $E_p\subset\O_p^{\br}$ is a finite set of transcendental elements over $\Q$ for each $p\in\PP$.
\end{enumerate}
 In this case, $\uE$ is polynomially factorizable.
\end{Prop}

\begin{proof}
(1) $\Rightarrow$ (2) This is an immediate consequence of {Theorem \ref{characterization local Krull domain}(2), applied to the localizations $R_{\Z\setminus p\Z}$ (Lemma \ref{lemma43})}.

(2) $\Rightarrow$ (3) Let $\{W_i\}_{i\in I}$ be the proper subset of the defining family of DVRs of $R$ which are extensions of $\Z_{(p)}$ for some $p\in\PP$. By Theorem \ref{DVR ra}, if $W_i$ is a residually algebraic extension of $\Z_{(p)}$, then there exists $\alpha_p\in\O_p^{\br}$, transcendental over $\Q$, such that $W_i=\Z_{(p),\alpha_p}=\Z_{(p),\alpha_p,\infty}$. For each $p\in\PP$, we denote by $E_p$ the subset formed by these $\alpha_p\in\O_p^{\br}$. Then   the defining family of DVRs of $R$ is equal to:
$$\{\Z_{(p),\alpha_p}\mid p\in\PP,\alpha_p\in E_p\}\cup\{\Q[X]_{(g)}\mid g\in\Q[X], \text{ irreducible}\}$$

In particular, we have
$$R=\bigcap_{i\in I}W_i\cap\Q[X]=\bigcap_{p\in\PP}\bigcap_{\alpha_p\in E_p}\Z_{(p),\alpha_p}\cap\Q[X]=\bigcap_{p\in\PP}\IntQ(E_p,\O_p)$$
where the last equality follows from Remark \ref{balls and intval rings}. Note that we also have
$$\bigcap_{p\in\PP}\IntQ(E_p,\O_p)=\IntQ(\uE,\O)=\{f\in\Q[X]\mid f(\alpha)\in\O,\forall\alpha\in \uE\}$$
where $\uE=\prod\limits_p E_p$, $\O=\prod\limits_p \O_p$ and $f(\alpha)=(f(\alpha_p))_p$ if $\alpha=(\alpha_p)_p\in\uE$.
 Finally, let $p\in\PP$. Then, since $p\in R$, $p$ is contained in finitely many centers of the defining family of DVRs of $R$. This implies that $E_p$ is finite.

(3) $\Rightarrow$ (1) Let $M$ be a maximal ideal of $R$.
If $M \cap \Z = (0)$, then {$\operatorname{ht}(M)=1$}, so assume that $M \cap \Z = p\Z$
for some $p \in \PP$. Note that
$R = \IntQ(\uE,\O) = \bigcap\limits_{p\in\PP}\bigcap\limits_{\alpha_p\in E_p}\Z_{(p),\alpha_p}\cap\Q[X]$, so if $S = \Z \setminus p\Z$, then
$R_S = \bigcap\limits_{\alpha_p\in E_p}\Z_{(p),\alpha_p}\cap\Q[X]$
by Lemma \ref{lemma43} and $\bigcap\limits_{\alpha_p\in E_p}\Z_{(p),\alpha_p}\cap\Q[X]$ is a Dedekind domain by Theorem \ref{thm4.23}. Hence, {$\operatorname{ht}(M)=\operatorname{ht}(MR_S) =1$}.
Thus, $R$ is a one-dimensional Krull domain, therefore $R$ is a Dedekind domain.

Now for the proof of ``In this case", let $f\in\Z[X]\subset R$ be nonzero.  Since $R$ is a Krull domain, $f$ is contained in only finitely many centers of the valuation domains of the defining family for $R$ (clearly,  $f$ is contained in finitely many centers of the non-unitary valuation domains $\Q[X]_{(g)}$, which are the non-unitary prime ideals of $R${).} This precisely means that the set ${\PP_{f,\uE}}=\{p\in\PP\mid \exists\alpha_p\in E_p, v_p(f(\alpha_p))>0\}$ is finite.  Hence, $\uE$ is polynomially factorizable (Definition \ref{def:polfact}).
\end{proof}

In particular, a polynomial Krull domain which is not Dedekind must have an essential valuation
overring which is a residually transcendental extension of $\Z_{(p)}$, for some $p\in\PP$.

\begin{Def}
{We say that a polynomial Krull domain $R$ is a \emph{pure} Krull domain if $Q$ is not maximal for all $Q \in X^1(R)$ with $Q \cap \Z \neq (0)$.}
\end{Def}

\begin{Rem}\label{Krull intersection}
{Let $R$ be a polynomial Krull domain, and let $R_1$ (resp., $R_2$) be the intersection of $\Q[X]$ with the unitary DVRs of the defining family of $R$ which are residually algebraic (resp., residually transcendental) over $\Z_{(p)}$ for some $p\in\PP$, with $R_i=\Q[X]$ if there are no such DVRs. Then $R=R_1\cap R_2$, $R_1$ is a polynomial Dedekind domain by Proposition \ref{Krull Dedekind}, and $R_2$ is either $\Q[X]$ or a pure Krull domain which is not Pr\"ufer, by \cite[Theorem 3.14]{PerPrufer} and Theorem \ref{characterization local Krull domain}(2) applied to the localizations $(R_2)_{\Z\setminus p\Z}$. This decomposition will be used in \S\ref{section: construction Krull}.}
\end{Rem}

\subsection{Construction of Krull domains with prescribed class group}\label{section: construction Krull}

By {Remark} \ref{Krull intersection}, if $R$ is a polynomial Krull domain that is not Dedekind, then $R$ can be represented as an intersection of a Dedekind domain $R_1$ and a pure Krull domain $R_2$, i.e., $R = R_1 \cap R_2$.
The divisor class group of a polynomial Krull domain is a direct sum of a countable family of finitely generated abelian groups by Theorem \ref{main thm polynomial krull domains}. Conversely, if $G$ is a direct sum of a countable family of finitely generated abelian groups, then there is a Dedekind domain between $\Z[X]$ and $\Q[X]$ with ideal class group $G$ \cite[Theorem 3.1]{PerDedekind}.

We next construct a pure Krull domain with prescribed divisor class group.

\begin{Prop}\label{existence Krull local case}
Let $G$ be a finitely generated abelian group. Then there exists a pure Krull domain $R$ between $\Z[X]$ and $\Q[X]$ such that $\operatorname{Cl}(R)=G$.
\end{Prop}

\begin{proof}
Suppose $G=\Z/e_1\Z \oplus \cdots \oplus \Z/e_n\Z \oplus \Z^m$
for some integers $e_1, \dots , e_n \in \N$ and $m \geq 0${; adding a trivial summand $\Z/1\Z$ if necessary, we may assume that $n\geq1$}.
Let $p_1, \dots , p_n \in \PP$ be distinct.

We pick $m+1$ distinct elements $\alpha_{p_n, 1},\ldots,\alpha_{p_n, m+1}$ in $\Z_{p_n}$ which are transcendental over $\Q$.  Let $N=\max\{v_{p_n}(\alpha_{p_n,i}-\alpha_{p_n, j})\mid i\not=j\in\{1,\ldots,m+1\}\}${ (with $N=0$ if $m=0$)}.
Let $W_{p_n, j} = \Z_{(p_n),\alpha_{p_n, j}, N+1/e_n}$, for $j\in\{1,\ldots,m +1\}$.  Note that the $W_{p_n,j}$'s are distinct (because $v_{p_n}(\alpha_{p_n, i}-\alpha_{p_n,j})<N+1/e_n$ for each
$i\not=j$) and $(\alpha_{p_n,j}, N+ 1/e_n)$ is a minimal pair for each
$j\in\{1,\ldots,m+1\}$. Then by {Theorem \ref{characterization local Krull domain}(3)} and Proposition \ref{unramified implies k=1} we have
 $$e_{p_n,j}=e(W_{p_n,j}\mid\Z_{(p_n)})=e(N+1/e_n\mid{\Q_{p_n}})=e_n.$$
By the same argument, we can also choose $\alpha_{p_i}$ in $\Z_{p_i}$
which is transcendental over $\Q$, so that $(\alpha_{p_i}, 1/e_i)$ is a minimal pair for $i=1, \dots, n-1$. Let $W_{p_i} = \Z_{(p_i), \alpha_{p_i}, 1/e_i}$ for $i =1, \dots, n-1$  so that, if we set {$R= \bigcap\limits_{i=1}^{n-1}W_{p_i} \cap \bigcap\limits_{j=1}^{m+1}W_{p_n,j}\cap\Q[X]$},
we have that $R$ is a Krull domain such that $\operatorname{Cl}(R)=G$
by Theorem \ref{main thm polynomial krull domains}
and $Q$ is not maximal for all $Q \in X^1(R)$ with $Q \cap \Z \neq (0)$
by Theorem \ref{characterization local Krull domain}(2).
\end{proof}

\begin{Thm}\label{thm:prescribed}
Let $G$ be a direct sum of a countable family $\{G_i \mid i \in I\}$ of finitely generated abelian groups (which are not necessarily distinct).
Then there exists a pure Krull domain $R$ between $\Z[X]$ and $\Q[X]$ with $\operatorname{Cl}(R)=G$. Moreover, for each $i \in I$, there is a multiplicative subset $S_i$ of $\Z$ such that $R_{S_i}$ is a pure Krull domain with $\operatorname{Cl}(R_{S_i}) = G_i$.
\end{Thm}

\begin{proof}
For each $i\in I$, let $G_i=\Z/e_{i1}\Z \oplus \cdots \oplus \Z/e_{in_i}\Z \oplus \Z^{m_i}$ for some $e_{i1}, \dots, e_{in_i} \geq 1$, $n_i \geq 1$ and $m_i\geq 0$. {Since $\PP$ is infinite and $I$ is countable, we may choose pairwise disjoint finite subsets $\PP_i\subseteq\PP$, $i\in I$, with $|\PP_i| = n_i$ for each $i\in I$.} By Proposition \ref{existence Krull local case}, there exists a pure Krull domain $R_i$, $\Z[X]\subseteq R_i\subset \Q[X]$, such that $\operatorname{Cl}(R_i)=G_i$.

Now set $R=\bigcap\limits_{i\in I}R_i$. Note that if $S_i = \Z \setminus \bigcup\limits_{p \in \PP_i}p\Z$ for each $i \in I$, then $R_{S_i} = R_i$ by Lemma \ref{lemma43}, thus
$R_{S_i}$ is a pure Krull domain with $\operatorname{Cl}(R_{S_i}) = G_i$.
However, it may be the case that $R$ is not a Krull domain, that is, it {does not have} the finite character property.

Let $i\in I$ and let $p$ be a fixed prime in the finite set $\PP_i$. Suppose that for the construction of $R_i$ we chose in $\Z_p$ $k_p\geq 1$ distinct elements, say for example $E_p=\{\alpha_{p,1},\ldots,\alpha_{p,k_p}\}$; let also $\delta_{p,j}$ be the radius of the ultrametric ball of center $\alpha_{p,j}$, for $j=1,\ldots,k_p$. We then   replace the elements in $E_p$ by possibly other  $k_p$ distinct $\alpha_p$'s in $\Z_p$ which lie in the same residue class $\lfloor\log(p)\rfloor+p\Z_{p}$, such that these $\alpha_p$'s are transcendental over $\Q$. We also modify the corresponding $\delta_{p,j}$'s accordingly so that the previous facts about the corresponding ring $R_i$ still hold.  We do this for every $i\in I$, including those for which $G_i$ is trivial (in this case $k_p=1$ and $\delta_{p,1}=1$ by the proof of Proposition \ref{existence Krull local case}), and we set $\Lambda=\bigcup_{i\in I}\PP_i$.  We claim that the set
{$$C(\underline E)=\prod_{p\in \Lambda}\{\alpha_{p,j}\mid 1\leq j\leq k_p,\ \delta_{p,j}>0\}$$}
is polynomially factorizable, so by Theorems \ref{Krull domains finite character} and \ref{characterization local Krull domain}(2), $R$ is a pure Krull domain, and $\operatorname{Cl}(R)=\bigoplus_{i\in I}\operatorname{Cl}(R_i)=G$ by Theorem \ref{main thm polynomial krull domains}. Let $g\in\Z[X]$  be nonzero, let $p\in\Lambda$ and let $\alpha_p\in C(E_p)$. Then we have
$$g(\alpha_p)\equiv g(\lfloor\log(p)\rfloor)\pmod{p}$$
Since for $p$ sufficiently large, $g(\lfloor\log(p)\rfloor)$ is a nonzero integer of absolute value less than $p$, hence it is not divisible by $p$, we get that $\PP_{g,C(\underline{E})}$ is finite.
\end{proof}

{\small
\subsection*{Acknowledgments}
This research was completed while the first author visited {the University} of Padova during January 2026. He thanks the Department of Mathematics for its hospitality. The first author was supported by Basic Science Research Program through the National Research Foundation of Korea (NRF) funded by the Ministry of Education (2017R1D1A1B06029867). The second author  is a member of the National Group for Algebraic and Geometric Structures and their Applications (GNSAGA) of the Italian Mathematics Research Institute (INdAM).}

\end{document}